\documentclass[12pt]{article}
\usepackage{amsmath,amsthm,amssymb, times, titlesec}
\usepackage{graphicx}
\usepackage{pdfsync}
\numberwithin{equation}{section}

\textheight 
            650pt
\textwidth 
           460pt

\titleformat{\subsubsection}
{\normalfont\fontsize{13}{14}\selectfont\itshape}{\thesubsubsection}{1em}{}

\newcommand{\nn}{|{\mskip-2mu}|{\mskip-2mu}|}

\newcommand{\R}{\mathbb{R}}
\newcommand{\Z}{\mathbb{Z}}
\newcommand{\N}{\mathbb{N}}

\newcommand{\ee}{\mathrm{e}}
\newcommand{\ii}{\mathrm{i}}

\newcommand{\dk}{\, \mathrm{d}k}
\newcommand{\ds}{\, \mathrm{d}s}
\newcommand{\dx}{\, \mathrm{d}x}
\newcommand{\dy}{\, \mathrm{d}y}
\newcommand{\dz}{\, \mathrm{d}z}

\renewcommand{\AA}{{\mathcal A}}
\newcommand{\BB}{{\mathcal B}}
\newcommand{\EE}{{\mathcal E}}
\newcommand{\FF}{{\mathcal F}}
\newcommand{\GG}{{\mathcal G}}
\newcommand{\HH}{{\mathcal H}}
\newcommand{\II}{{\mathcal I}}
\newcommand{\JJ}{{\mathcal J}}
\newcommand{\KK}{{\mathcal K}}
\newcommand{\LL}{{\mathcal L}}
\newcommand{\NN}{{\mathcal N}}
\newcommand{\OO}{{\mathcal O}}
\newcommand{\QQ}{{\mathcal Q}}
\newcommand{\RR}{{\mathcal R}}
\newcommand{\TT}{{\mathcal T}}
\newcommand{\XX}{{\mathcal X}}
\newcommand{\YY}{{\mathcal Y}}
\newcommand{\ZZ}{{\mathcal Z}}

\DeclareMathOperator{\supp}{supp}

\newtheorem{theorem}{Theorem}
\newtheorem{lemma}{Lemma}
\newtheorem{proposition}{Proposition}
\newtheorem{corollary}{Corollary}

\theoremstyle{definition}
\newtheorem{remark}{Remark}
\newtheorem*{note}{Note}

\newcounter{count}

\title{A variational reduction and the existence of a fully localised solitary wave for the three-dimensional water-wave problem with weak surface tension}

\author{
B. Buffoni\thanks{Section de math\'ematiques, Station 8, Ecole Polytechnique
F\'ed\'erale de Lausanne, 1015 Lausanne, Switzerland
} 
\and M. D. Groves\thanks{Fachrichtung Mathematik, Universit\"{a}t des Saarlandes,
Postfach 151150, 66041 Saarbr\"{u}cken, Germany; 
Department of Mathematical Sciences, Loughborough
University, Loughborough, LE11 3TU, UK
} 
\and E. Wahl\'en\footnote{Centre for Mathematical Sciences, Lund University, P.O. Box 118, 22100 Lund, Sweden}
}
\date{}

\begin{document}

\maketitle

\begin{abstract}
\emph{Fully localised solitary waves} are travelling-wave solutions
of the three-dimensional gravity-capillary water wave problem which
decay to zero in every horizontal spatial direction. Their existence
has been predicted on the basis of numerical simulations and
model equations (in which context they are usually referred to
as `lumps'), and a mathematically rigorous existence theory for
strong surface tension (Bond number $\beta$ greater than $\frac{1}{3}$)
has recently been given. In this article we present an existence theory
for the physically more realistic case $0<\beta<\frac{1}{3}$. A classical variational
principle for fully localised solitary waves
is reduced to a locally equivalent variational principle featuring
a perturbation of the functional associated with the Davey-Stewartson
equation. A nontrivial critical point of the reduced functional is found
by minimising it over its natural constraint set.
\end{abstract}

\section{Introduction}

\subsection{The hydrodynamic problem}

The classical \emph{three-dimensional gravity-capillary water wave
problem} concerns the irrotational flow of a perfect fluid of unit
density subject to the forces of gravity and surface tension. The
fluid motion is described by the Euler equations in a domain bounded
below by a rigid horizontal bottom $\{y=0\}$ and above by a free
surface $\{y=1+\eta(x,z,t)\}$, where the
function $\eta$ depends upon the two horizontal spatial directions
$x$, $z$ and time $t$.
In terms of an Eulerian velocity potential $\varphi$,
the mathematical problem is to solve Laplace's equation
\begin{equation}
\varphi_{xx} + \varphi_{yy} + \varphi_{zz} = 0, \qquad\qquad 0<y<1+\eta, \label{SWW 1}
\end{equation}
with boundary conditions
\begin{eqnarray}
\varphi_{y} & = & \parbox{100mm}{$0$,} y=0, \label{SWW 2} \\
\eta_t & = & \parbox{100mm}{$\varphi_{y} - \eta_x\varphi_x - \eta_z\varphi_z$,}  y=1+\eta, \label{SWW 3}
\\
\varphi_t & = & -\frac{1}{2}(\varphi_x^2+\varphi_{y}^2+\varphi_z^2) -\eta \nonumber \\
& &
\parbox{100mm}{$\displaystyle\qquad\mbox{} + \beta\left[\frac{\eta_x}{\sqrt{1+\eta_x^2+\eta_z^2}}\right]_x
+ \beta\left[\frac{\eta_z}{\sqrt{1+\eta_x^2+\eta_z^2}}\right]_z$,} y=1+\eta. \label{SWW 4}
\end{eqnarray}
Note that we use dimensionless variables, choosing $h$ as length scale,
$(h/g)^\frac{1}{2}$ as time scale and introducing the Bond number $\beta=\sigma/gh^2$,
where $h$ is the depth of the water in its undisturbed state, $g$ is the acceleration due to
gravity and $\sigma>0$ is the coefficient of surface tension.
In this article we consider \emph{fully localised solitary waves}, that is travelling-wave solutions to
\eqref{SWW 1}--\eqref{SWW 4} of the form
$\eta(x,z,t)=\eta(x+c t,z)$, $\varphi(x,y,z,t)=\varphi(x+c t,y,z)$
(so that the waves move with unchanging shape and constant speed $c$ from right to left)
with $\eta(x+c t,z) \rightarrow 0$ as $|(x+c t,z)| \rightarrow \infty$ (so that the waves
decay in every horizontal direction). We always take $\beta$ in the interval $(0,\frac{1}{3})$
(`weak surface tension').

To formulate our main result, let us first note that the function $s \mapsto c(s)$, $s \geq 0$ given by
$c(s)=\sqrt{(1+\beta s^2)/(s \coth s)}$ (the linear dispersion relation for a two-dimensional travelling wave train
with wave number $s \geq 0$ and speed $c>0$ -- see Figure \ref{Dispersion relation} below) has a unique global minimum at $s=\omega>0$;
denote the minimum value of $c^2$ by $\Lambda$.

\begin{theorem} \label{Main theorem}
Suppose that $0<\beta<\frac{1}{3}$ and $c^2=(1-\varepsilon^2)\Lambda$. There exists a fully localised solitary-wave
solution \eqref{SWW 1}--\eqref{SWW 4} for each sufficiently small value of $\varepsilon>0$.
\end{theorem}

This result confirms the prediction made on the basis of model equations, in particular the 
Davey-Stewartson equation (see Djordjevic \& Redekopp \cite{DjordjevicRedekopp77}, Ablowitz \& Segur
\cite{AblowitzSegur79} and Cipolatti \cite{Cipolatti92}), and numerical computations by
Parau, Vanden-Broeck \& Cooker \cite{ParauVandenBroeckCooker05a}
(see Figure \ref{Fully localised solitary wave} for a sketch of a typical free surface in their simulations).
It also complements recent existence theories for $\beta>\frac{1}{3}$ (`strong surface tension')
by Groves \& Sun \cite{GrovesSun08} and Buffoni \emph{et al.} \cite{BuffoniGrovesSunWahlen13}
(which also confirm prediction made by model equations, in particular the KP-I equation -- see Kadomtsev \& Petviashvili
\cite{KadomtsevPetviashvili70} and Ablowitz \& Segur \cite{AblowitzSegur79}).

\begin{figure}[h]
\centering

\includegraphics[width=8cm]{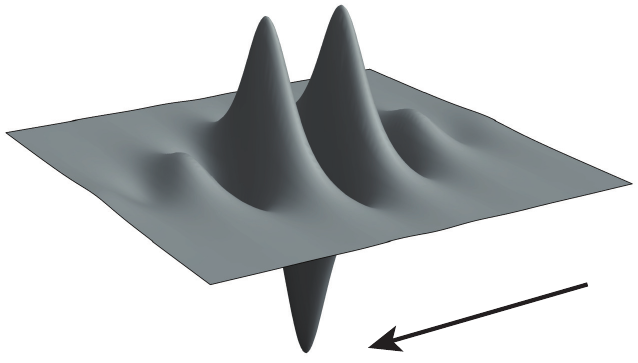}
{\it
\caption{Sketch of a fully localised solitary wave with weak surface tension; the
arrow shows the direction of wave propagation. \label{Fully localised solitary wave}}}
\end{figure}

\subsection{A variational principle}

The proof of Theorem \ref{Main theorem} is variational in nature.
Fully localised solitary waves are characterised as critical points of the wave energy
$$\EE(\eta,\varphi) = \int_{\R^2}\left\{ \frac{1}{2}\int_0^{1+\eta} (\varphi_x^2+\varphi_y^2+\varphi_z^2) \dy
+ \frac{1}{2}\eta^2 + \beta[\sqrt{1+\eta_x^2+\eta_z^2}-1]\right\}\dx\dz$$
subject to the constraint that the momentum
$$\II(\eta,\varphi) = \int_{\R^2} \eta_x \varphi|_{y=1+\eta} \dx\dz$$
in the $x$-direction is fixed (both are conserved quantities of \eqref{SWW 1}--\eqref{SWW 4} -- see articles by Zakharov \& Kuznetsov \cite{Za68,ZaKu74,ZaKu84,ZaKu97} and
Benjamin \& Olver \cite{BenjaminOlver82}); the wave speed $c$ is the Lagrange multiplier in the variational principle
$\delta(\EE-c\II)=0$. More
satisfactory representations of these functionals are obtained by means of the Dirichlet-Neumann operator
$G(\eta)$ introduced by Craig \cite{Craig91} and defined as follows.
For fixed $\Phi=\Phi(x,z)$ solve the boundary-value problem
\begin{eqnarray*}
& & \parbox{6cm}{$\varphi_{xx}+\varphi_{yy}+\varphi_{zz}=0,$}0<y<1+\eta, \\
& & \parbox{6cm}{$\varphi = \Phi,$}y=1+\eta, \\
& & \parbox{6cm}{$\varphi_y =0,$}y=0
\end{eqnarray*}
and define
\begin{align*}
G(\eta)\Phi & = \sqrt{1+\eta_x^2+\eta_z^2}\left.\frac{\partial \varphi}{\partial n}\right|_{y=1+\eta} \\
& = \varphi_y - \eta_x\varphi_x - \eta_z\varphi_z\Big|_{y=1+\eta}.
\end{align*}
Working with the variables $\eta=\eta(x,z)$ and $\Phi=\Phi(x,z)$, one finds that
$$
\EE(\eta,\Phi) = \int_{\R^2}\left\{ \frac{1}{2}\Phi G(\eta)\Phi
+ \frac{1}{2}\eta^2 + \beta[\sqrt{1+\eta_x^2+\eta_z^2}-1]\right\}\dx\dz,
$$
$$
\II(\eta,\Phi) = \int_{\R^2} \eta_x \Phi \dx\dz. 
$$

We find nontrivial critical points of $\EE-c\II$ in two steps.
(i) For given $\eta \neq 0$, we observe that
$\EE(\eta,\cdot)-c\II(\eta,\cdot)$ has a unique critical
point $\Phi_\eta$ which is the 
unique global minimiser $\Phi_\eta$ of $\EE(\eta,\cdot)-c\II(\eta,\cdot)$
and satisfies  $G(\eta)\Phi_\eta = c \eta_x$.
(ii) We seek nontrivial critical points of the functional
\begin{align}
\JJ(\eta)&:=\EE(\eta,\Phi_\eta)-c\II(\eta,\Phi_\eta) \nonumber \\
&=\EE(\eta,c G(\eta)^{-1}\eta_x)-c\II(\eta,cG(\eta)^{-1}\eta_x) \nonumber \\
&=\KK(\eta)-c^2 \LL(\eta), \label{Defn of J}
\end{align}
where
$$
\KK(\eta)=\int_{\R^2} \left(\frac12 \eta^2 +\beta\sqrt{1+\eta_x^2+\eta_z^2}-\beta\right)\dx\dz,
\qquad
\LL(\eta)=\frac{1}{2}\int_{\R^2} \eta\, K(\eta) \eta \dx\dz
$$
and $K(\eta)=-\partial_x G(\eta)^{-1}\partial_x$. The following theorem is a reformulation of our main result
in terms of critical points of $\JJ$.

\begin{theorem} \label{Variational main theorem}
Suppose that $0<\beta<\frac{1}{3}$ and $c^2=(1-\varepsilon^2)\Lambda$. The formula
\eqref{Defn of J} defines a smooth functional
$\JJ_\varepsilon: U \rightarrow \R$, where $U$ is a suitably chosen open neighbourhood of the origin in $H^3(\R^2)$,
which has a nontrivial critical point for each sufficiently
small value of $\varepsilon>0$.
\end{theorem}

\subsection{Variational reduction}

The existence of fully localised solitary waves with weak surface tension has been predicted by approximating
the hydrodynamic equations \eqref{SWW 1}--\eqref{SWW 4} by simpler model equations, in particular
the Davey-Stewartson equation (see Djordjevic \& Redekopp \cite{DjordjevicRedekopp77} and Ablowitz \& Segur
\cite{AblowitzSegur79}). Fully localised solitary wave solutions to the Davey-Stewartson equation have a variational
characterisation, and the direct methods of the calculus of variations have been used to show that it indeed has
such a solution
(see Cipolatti \cite{Cipolatti92} and Papanicolaou \emph{et al.} \cite[\S5]{PapanicolaouSulemSulemWang94}).
In this paper we seek critical points of the functional $\EE-c\II$. A direct application of well-developed standard variational methods, which are optimised for semilinear partial differential equations, is not possible due to the quasilinear nature of the hydrodynamic problem (see the discussion by Groves \& Sun \cite{GrovesSun08} and Buffoni et al. \cite{BuffoniGrovesSunWahlen13}). Instead we proceed by performing a rigorous local variational reduction (akin to the variational Lyapunov-Schmidt reduction)
which converts it to a perturbation of the Davey-Stewartson variational functional
(Section \ref{Reduction}). Critical points of the reduced functional are then found by applying
the direct methods of the calculus of variations in a perturbative fashion (Section \ref{sec:Existence theory}).

It is instructive to review the formal derivation of the Davey-Stewartson equation for travelling waves.
We begin with the linear dispersion relation for a two-dimensional sinusoidal
travelling wave train with wave number $s\geq0$ and speed $c>0$, namely
$$c^2 =\frac{1+\beta s^2}{f(s)}, \qquad f(s)=s \coth s$$
(see Figure \ref{Dispersion relation}).
For each fixed $\beta \in (0,\frac{1}{3})$ the function $s \mapsto c(s)$,
$s \geq 0$ has a unique global minimum at $s=\omega>0$ (the formula $\beta=f^\prime(\omega)/(2\omega f(\omega)-\omega^2f^\prime(\omega))$ defines a bijection
between the values of $\beta \in (0,\frac{1}{3})$ and $\omega \in (0,\infty)$);
we denote the minimum value of $c^2$ by $\Lambda$ (so that
$\Lambda=2\omega/(2\omega f(\omega)-\omega^2f^\prime(\omega))$).
Note for later use that
\begin{equation}
g(s):=1+\beta s^2 - \Lambda f(s) \geq 0, \qquad s \in \R, \label{Defn of g}
\end{equation}
with equality precisely when $s=\pm \omega$. 
Bifurcations of nonlinear solitary waves are expected whenever the
linear group and phase speeds are equal, so that $c^\prime(s)=0$ (see Dias \& Kharif \cite[\S 3]{DiasKharif99}).
We therefore expect the existence of small-amplitude solitary waves with speed near $\Lambda^{1/2}$;
the waves bifurcate from a linear periodic wave train with wavenumber $\omega$.

\begin{figure}[h]
\centering

\includegraphics[scale=0.67]{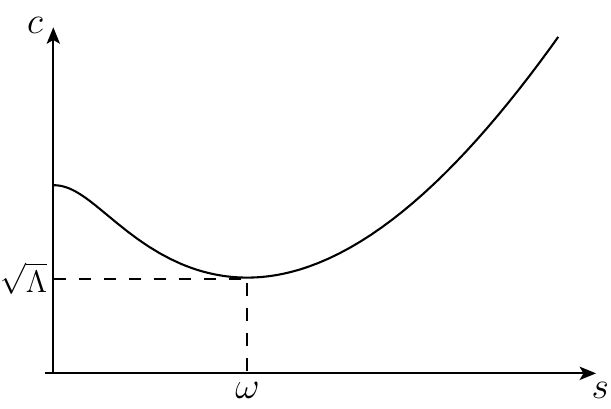}
{\it
\caption{Dispersion relation for a two-dimensional travelling wave train with wave number $s \geq 0$ and speed $c>0$.\label{Dispersion relation}}}
\end{figure}

Making the travelling-wave \emph{Ansatz} $\eta(x,z,t)=\tilde{\eta}(x+ct,z)$ and substituting
$c=\sqrt{(1-\varepsilon^2)\Lambda}$,
\begin{equation}
\tilde{\eta}(x,z) = \frac{1}{2}\varepsilon\zeta(\varepsilon x, \varepsilon z) \mathrm{e}^{\mathrm{i} \omega x}+\frac{1}{2}\varepsilon
\overline{\zeta(\varepsilon x, \varepsilon z)} \mathrm{e}^{-\mathrm{i} \omega x} \label{DS Ansatz}
\end{equation}
into equations \eqref{SWW 1}--\eqref{SWW 4}, one finds that
to leading order $\zeta$ satisfies the Davey-Stewartson equation
$$
- a_1 \zeta_{XX}-a_2 \zeta_{ZZ} + a_3 \zeta - 2C_1 \FF^{-1}\left[\frac{k_1^2}{(1-\Lambda)k_1^2+k_2^2} \FF[|\zeta|^2]\right]\zeta
-2C_2|\zeta|^2\zeta =0,
$$
where $X=\varepsilon x$, $Z=\varepsilon z$,
$$a_1=\frac{1}{8}\partial_{k_1}^2\tilde{g}(\omega,0), \qquad a_2=\frac{1}{8}\partial_{k_2}^2\tilde{g}(\omega,0), \qquad a_3=\frac{1}{4}\Lambda f(\omega),$$
$$\tilde{g}(k)=g(|k|) +\Lambda \frac{k_2^2}{|k|^2}f(|k|)$$
and formulae for the positive coefficients $C_1$, $C_2$
are given in Theorem \ref{thm:red func} (see Djordjevic \& Redekopp \cite{DjordjevicRedekopp77} and
Ablowitz \& Segur \cite{AblowitzSegur79}, noting the misprint in equation
(2.24d)). The Davey-Stewartson equation is the Euler-Lagrange equation for the functional
$\TT_0: H^1(\R^2) \rightarrow \R$ given by the formula
\begin{eqnarray*}
\lefteqn{\TT_0(\zeta) = \int_{\R^2}\big(a_1|\zeta_x|^2+a_2|\zeta_z|^2+a_3|\zeta|^2\big)\dx\dz} \qquad\qquad\\
& & \mbox{}-C_1\int_{\R^2}\frac{k_1^2}{(1-\Lambda)k_1^2+k_2^2}
|\FF[|\zeta|^2]|^2\dk_1\dk_2
-C_2\int_{\R^2}|\zeta|^4\dx\dz,
\end{eqnarray*}
where $\FF$ and $\FF^{-1}$ denote the Fourier and inverse Fourier transforms and we have replaced $(X,Z)$ with $(x,z)$.
This functional has a nontrivial critical point (Cipolatti \cite{Cipolatti92}, Papanicolaou \emph{et al.} \cite[\S5]{PapanicolaouSulemSulemWang94}), which of course corresponds to a
fully localised solitary-wave solution of the Davey-Stewartson equation (often called a `lump' solution).

Let us now return to the water-wave problem and in particular the task of finding a nontrivial critical point of
the functional
$$\JJ_\varepsilon(\eta)=\KK(\eta) - \Lambda(1-\varepsilon^2)\LL(\eta);$$
we study $\JJ_\varepsilon$ in a suitably chosen
neighbourhood $U$ of the origin in its function space $H^3(\R^2)$.
The Ansatz \eqref{DS Ansatz} suggests that the spectrum of a fully localised solitary wave
is concentrated near the points $(\omega,0)$ and $(-\omega,0)$.
We therefore decompose $\eta$ into the sum of functions $\eta_1$ and $\eta_2$ whose Fourier transforms
$\hat{\eta}_1$ and $\hat{\eta}_2$ are supported in
the region $S=B_\delta(\omega,0) \cup B_\delta(-\omega,0)$ (with $\delta \in (0,\frac{\omega}{3})$)
and its complement (see Figure \ref{Splitting}), so that $\eta_1 = \chi(D)\eta$,
$\eta_2 = (1-\chi(D))\eta$, where $\chi$ is the characteristic function of the set $S$ and $\chi(D)$ denotes the Fourier-multiplier
operator with symbol $\chi$.

\begin{figure}[h]
\centering

\includegraphics[scale=0.8]{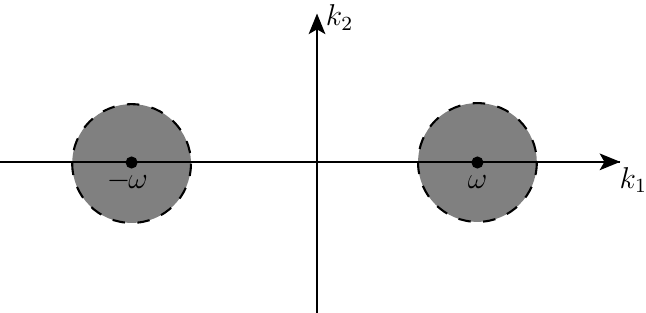}
{\it
\caption{The support of $\hat{\eta}_1$ is contained in the set $S=B_\delta(\omega,0) \cup B_\delta(-\omega,0)$. \label{Splitting}}}
\end{figure}

Observe that
$\eta \in U$ is a critical point of $\JJ_\varepsilon$, that is
$$
\mathrm{d}\JJ_\varepsilon[\eta](v)=0
$$
for all $v \in \XX:=H^3(\R^2)$, if and only if
$$\mathrm{d}\JJ_\varepsilon[\eta_1+\eta_2](v_1)=0, \qquad \mathrm{d}\JJ_\varepsilon[\eta_1+\eta_2](v_2)=0,$$
for all $v_1 \in \XX_1:=\chi(D)\XX$ and $v_2 \in \XX_2:=(1-\chi(D))\XX$.
For sufficiently small values of $\varepsilon>0$ the second of these equations can be solved for $\eta_2$ as
a function of $\eta_1$, and we thus obtain the reduced functional
$$\tilde{\JJ_\varepsilon}(\eta_1)=\JJ_\varepsilon(\eta_1+\eta_2(\eta_1)).$$
Applying the Davey-Stewartson scaling \eqref{DS Ansatz} to $\eta_1$, one obtains a reduced equation for $\zeta$
which is the Euler-Lagrange equation for the functional $\TT_\varepsilon: U_\varepsilon \rightarrow \R$
given by
$$
\TT_\varepsilon(\zeta)=\TT_0(\zeta)+O(\varepsilon^{1/2}\|\zeta\|_1^2)
$$
(with corresponding estimates for the derivatives of the remainder term);
each critical point $\zeta_\infty$ of $\TT_\varepsilon$ with $\varepsilon>0$
corresponds to a critical point $\eta_\infty$ of
$\widetilde\JJ_\varepsilon$, which in turn defines a critical point $\eta_1+\eta_2(\eta_1)$ of $\JJ_\varepsilon$.
Here $U_\varepsilon:=B_R(0)\subset H_\varepsilon^1(\R^2):=\chi(\varepsilon D)H^1(\R^2)$, where $R$
is independent of $\varepsilon$ and can be chosen arbitrarily large.

All estimates are given in Section \ref{Reduction}
are uniform over values of $\varepsilon$
in an interval $(0,\varepsilon_0)$ and in general we
replace $\varepsilon_0$ with a smaller number if necessary for the validity of our results (note in
particular that $\varepsilon_0 \rightarrow 0$ as $R \rightarrow \infty$). We consistently abbreviate
inequalities of the form $g_1(s) \leq k g_2(s)$, where $k$ is a
generic constant which does not depend upon $\varepsilon$, to $g_1(s) \lesssim g_2(s)$.

\begin{remark}
The dispersion relation for surface waves on water of infinite depth
also exhibits the features shown in Figure \ref{Dispersion relation}, and the corresponding
travelling-wave Ansatz leads to the two-dimensional nonlinear Schr\"{o}dinger equation.
The dispersion relation for strong surface tension ($\beta>\frac{1}{3}$) is however qualitatively different,
having a unique global minimum at $s=0$ (with $\Lambda=1$); in this case the
Ansatz $\tilde{\eta}(x,z)=\varepsilon^2 \zeta(\varepsilon x, \varepsilon^2 z)$ leads to the KP-I equation.
The two-dimensional nonlinear Schr\"{o}dinger and KP-I equations
have variational characterisations and `lump' solutions, and the
variational reduction of the water-wave problem to a perturbation of one of these equations will be
discussed elsewhere.
\end{remark}

\subsection{Critical points of the reduced functional}

\begin{figure}[h]
\centering
\includegraphics[width=4cm]{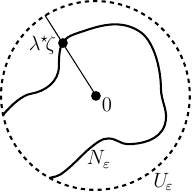}
\hspace{2cm}
\includegraphics[width=5cm]{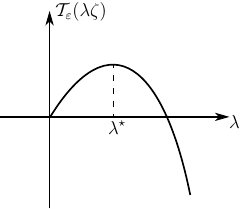}

{\it
\caption{Any ray intersects the natural constraint manifold $N_\varepsilon$
in at most one point and the value of $\TT_\varepsilon$ along such a ray attains a strict maximum at this point.}
\label{ncs geometry}}
\end{figure}

In Section \ref{sec:Existence theory} we seek critical points of $\TT_\varepsilon$ by minimising it on its
\emph{natural constraint set}
$$N_\varepsilon=\{\zeta\in U_\varepsilon:\zeta\neq 0,\mathrm{d}\TT_\varepsilon[\zeta](\zeta)=0\},$$
our motivation being the observation that any \emph{ground state}, that is a (necessarily nontrivial) minimiser
of $\TT_\varepsilon$ over $N_\varepsilon$, is a critical point of $\TT_\varepsilon$
(see Remark \ref{rem:ground states} and Willem \cite[\S 4]{Willem} for a general discussion of natural
constraint sets). The natural constraint set has a geometrical interpretation (see Figure \ref{ncs geometry}),
namely that
any ray in $B_R(0)\subset H_\varepsilon^1(\R^2)$ intersects the natural constraint manifold $N_\varepsilon$ in at most one point and the value of $\TT_\varepsilon$ along such a ray attains a strict maximum at this point.
This fact is readily established by a direct calculation for $\varepsilon=0$ and deduced by a perturbation argument for $\varepsilon>0$, and similar perturbative methods yield the existence of a 
a minimising sequence $\{\zeta_n\} \subset B_{R-1}(0)$
with
$$\TT_\varepsilon|_{N_\varepsilon} \rightarrow \inf \TT_\varepsilon|_{N_\varepsilon}>0, \qquad
\mathrm{d}\TT_\varepsilon[\zeta_n]\rightarrow 0$$ as $n \rightarrow \infty$.

We study minimising sequences of the above kind in Section \ref{Existence 1}, where the following
theorem is established by a weak continuity argument.

\begin{theorem} \label{Variational 1}
Let $\{\zeta_n\} \subset B_{R-1}(0)$ be a minimising sequence for $\TT_\varepsilon|_{N_\varepsilon}$
with
$\mathrm{d}\TT_\varepsilon[\zeta_n]\rightarrow 0$ as $n \rightarrow \infty$.
There exists a sequence $\{w_n\} \subset \Z^2$ with the property that
$\{\zeta_n(\cdot+w_n)\}$
converges weakly to a nontrivial critical point $\zeta_\infty$ of
$\TT_\varepsilon$.
\end{theorem}

The short proof of Theorem \ref{Variational 1} does not show that the critical point $\eta_\infty$ is a ground state. This deficiency is removed in Section \ref{Existence 2} with the help of an abstract version of the concentration-compactness principle (Lions \cite{Lions84a,Lions84b}) which is given in the Appendix.

\begin{theorem} \label{Variational 2}
Let $\{\zeta_n\} \subset B_{R-1}(0)$ be a minimising sequence for $\TT_\varepsilon|_{N_\varepsilon}$
with $\mathrm{d}\TT_\varepsilon[\zeta_n]\rightarrow 0$ as $n \rightarrow \infty$.
There exists a sequence $\{w_n\} \subset \Z^2$ with the property that
$\{\zeta_n(\cdot+w_n)\}$
converges weakly to a ground state $\zeta_\infty$.
\end{theorem}

We prove Theorems \ref{Variational 1} and \ref{Variational 2} for $\varepsilon>0$, taking advantage of the relationship between
the functionals $\JJ_\varepsilon$ and $\TT_\varepsilon$ and the fact that
the spaces $\chi(\varepsilon D)H^s(\R^2)$, $s \geq 0$ are all topologically equivalent;
the function $\eta_\infty = \eta_1(\zeta_\infty) + \eta_2(\eta_1(\zeta_\infty))$ given by these theorems
is then a nonzero critical point of $\JJ_\varepsilon$, which concludes the proof of Theorem \ref{Main theorem}. 

\begin{note}
The main results of this paper have been announced by Buffoni \cite{Buffoni2015}.
\end{note}

\section{Variational reduction} \label{Reduction}

\subsection{The variational functional} \label{VF}
In this section we discuss functional-analytic aspects of the functional
$$
\JJ_\varepsilon(\eta)=\KK(\eta) - \Lambda(1-\varepsilon^2)\LL(\eta),
$$
in which
\begin{equation}
\KK(\eta)=\int_{\R^2} \left(\frac12 \eta^2 +\beta\sqrt{1+\eta_x^2+\eta_z^2}-\beta\right)\dx\dz, \qquad
\LL(\eta)=\frac{1}{2}\int_{\R^2} \eta\, K(\eta) \eta \dx\dz
\label{Defn of KK,LL}
\end{equation}
and $K(\eta)\xi=-(\varphi|_{y=1+\eta})_x$, where $\varphi$ is the solution of the boundary-value problem
\begin{align}
&\Delta \varphi=0, && 0< y < 1+\eta(x,z), \label{Kproblem 1} \\
&\varphi_y=0, && y=0, \label{Kproblem 2} \\
&\varphi_y-\eta_x \varphi_x-\eta_z \varphi_z=\xi_x, && y=1+\eta(x,z) \label{Kproblem 3}
\end{align} 
(which is unique up to an additive constant). 
We examine the boundary-value problem \eqref{Kproblem 1}--\eqref{Kproblem 3} below and show in particular that the mapping
$K(\cdot)\colon \ZZ \to \LL(H^{5/2}(\R^2), H^{3/2}(\R^2))$ is analytic at the origin (Corollary \ref{cor:K analytic}), where
$$
\ZZ=\{\eta \in \mathcal{S}^\prime(\R^2): \|\eta\|_\ZZ := \|\hat{\eta}_1\|_{L^1(\R^2)} + \|\eta_2\|_3 < \infty\},
$$
$$\eta_1=\chi(D) \eta, \qquad \eta_2=(1-\chi(D)) \eta,$$
and $\chi$ is the characteristic function of the set $S=B_\delta(\omega,0) \cup B_\delta(-\omega,0)$
(with $0<\delta<\frac{\omega}{3}$).
In view of this result we choose $M$ sufficiently small and study $\JJ_\varepsilon$ in
the set
$$U=\{\eta \in H^3(\R^2): \|\eta\|_\ZZ < M\},$$
noting that $H^3(\R^2)$ is continuously embedded in $\ZZ$ and that $U$ is an open neighbourhood of the origin in $H^3(\R^2)$.
(Here, and in the remainder of the paper, we denote the usual norm for $H^r(\R^2)$ by $\|\cdot\|_r$ and for $W^{m,\infty}(\R^2)$
by $\|\cdot\|_{m,\infty}$.)

\subsubsection*{The boundary-value problem (2.2)--(2.4)}

This boundary-value problem is studied using the change of variable
\begin{equation}
\label{flattening}
y^\prime=\frac{y}{1+\eta},\qquad u(x,y^\prime,z)=\varphi(x,y,z),
\end{equation}
which maps $\Sigma_\eta=\{(x,y,z)\colon x,z\in \R, 0<y<1+\eta(x,z)\}$ to the `slab'
$\Sigma=\{(x,y^\prime,z)\colon x,z\in \R, y^\prime\in (0,1)\}$.
Dropping the primes, we find that the boundary-value problem is transformed into
\begin{align}
&\Delta u=\partial_x F_1(\eta,u)+\partial_y F_3(\eta,u)+\partial_z F_2(\eta,u), && \quad 0 < y <1, \label{flatKproblem 1} \\
&u_y=0,&& \quad y=0, \label{flatKproblem 2} \\
&u_y=F_3(\eta,u)+\xi_x, && \quad y=1, \label{flatKproblem 3}
\end{align}
where
\[
F_1(\eta,u)=-\eta u_x+y\eta_x u_y, \quad
F_2(\eta,u)=-\eta u_z+y\eta_z u_y,\]
\[
F_3(\eta,u)=\frac{\eta u_y}{1+\eta}+y\eta_x u_x+y\eta_z u_z-\frac{y^2}{1+\eta}(\eta_x^2+\eta_z^2) u_y;
\]
equations \eqref{flatKproblem 1}--\eqref{flatKproblem 3} are studied using the following proposition, whose
proof is given by Buffoni \emph{et al.} \cite[Propositions 2.20 and 2.21]{BuffoniGrovesSunWahlen13}.

\begin{proposition} \label{prop:Defn of Gamma}
Suppose that $r \geq 1$. For each $F_1$, $F_2$, $F_3 \in H^r(\Sigma)$ and $\xi \in H^{r+1/2}(\R^2)$
the boundary-value problem
\begin{align*}
&\Delta u=\partial_x F_1+\partial_y F_3+\partial_z F_2, && \quad 0 < y <1, \\
&u_y=0,&& \quad y=0, \\
&u_y=F_3+\xi_x, && \quad y=1
\end{align*}
admits a solution $u$ which is unique up to an additive constant. Furthermore, the mapping
$(F_1,F_2,F_3,\xi) \mapsto u$ defines a bounded linear operator $\Gamma: H^r(\Sigma)^3 \times H^{r+1/2}(\R^2)
\rightarrow H^{r+1}_\star(\Sigma)$, where
$H_\star^{r+1}(\Sigma)$ is the completion of
$$\mathcal{S}(\Sigma,{\mathbb R})=\{u \in C^\infty(\bar{\Sigma}):
|(x,z)|^m|\partial_x^{\alpha_1}\partial_y^{\alpha_2}\partial_z^{\alpha_3}u|\mbox{ is bounded for all }m,\alpha_1,\alpha_2, \alpha_3 \in {\mathbb N}_0 \}$$
with respect to the norm
$$ \|u\|_{H^{r+1}_\star(\Sigma)}:=\|u_x\|_{H^r(\Sigma)} + \|u_y\|_{H^r(\Sigma)} + \|u_z\|_{H^r(\Sigma)}.$$
\end{proposition}

The following result is obtained by the method used by Buffoni \emph{et al.} \cite[Corollary 2.23 and Proposition 2.29]{BuffoniGrovesSunWahlen13}
(who work in the standard Sobolev space $H^3(\R^2))$,
where the `elementary inequality' quoted on page 1031 there is replaced by
\begin{align*}
\|\eta w\|_{H^2(\Sigma)} & \leq  \|\eta_1 w\|_{H^2(\Sigma)} + \|\eta_2 w\|_{H^2(\Sigma)}\\
& \lesssim \|\eta_1\|_{2,\infty} \|w\|_{H^2(\Sigma)} + \|\eta_2\|_2 \|w\|_{H^2(\Sigma)} \\
& \lesssim (\|\hat{\eta}_1\|_{L^1(\R^2)} + \|\eta_2\|_3) \|w\|_{H^2(\Sigma)} \\
&= \|\eta\|_\ZZ \|w\|_{H^2(\Sigma)}
\end{align*}
(which also holds for $\|\eta_x w\|_2$ and $\|\eta_z w\|_2$ since $\hat{\eta}_1$ has compact support).

\begin{lemma} \label{lem:u analytic}
For each $\xi \in H^{5/2}(\R^2)$ and sufficiently small $\eta \in \ZZ$
the boundary-value problem \eqref{flatKproblem 1}--\eqref{flatKproblem 3} admits a solution $u$ which is unique
up to an additive constant and satisfies $u \in H^3_\star(\Sigma)$.
Furthermore, the mapping $\ZZ\to \mathcal L(H^{5/2}(\R^2), H^3_\star(\Sigma))$ given by 
$\eta \mapsto (\xi \mapsto u)$ is analytic at the origin.
\end{lemma}
\begin{corollary} \label{cor:K analytic}
The mapping $K(\cdot)\colon \ZZ \to \LL(H^{5/2}(\R^2), H^{3/2}(\R^2))$ is analytic at the origin.
\end{corollary}

\subsubsection*{Analyticity of the functionals and their gradients in $L^2(\R^2)$}

\begin{corollary}
The formulae \eqref{Defn of KK,LL} define analytic functionals $\KK, \LL: U\rightarrow \R$.
\end{corollary}

Our next result is proved by combining Lemma \ref{lem:u analytic} with the calculation given in the proof
of Lemma 2.27 in Buffoni \emph{et al.} \cite{BuffoniGrovesSunWahlen13}.

\begin{lemma} \label{lem:gradients}
The gradients $\KK^\prime(\eta)$ and $\LL^\prime(\eta)$ in $L^2(\R^2)$
exist for each $\eta \in U$ and are given by the formulae
\begin{align*}
\KK^\prime(\eta) & = -\beta\left[\frac{\eta_x}{\sqrt{1+\eta_x^2+\eta_z^2}}\right]_{\!x}
-\beta\left[\frac{\eta_x}{\sqrt{1+\eta_x^2+\eta_z^2}}\right]_{\!z}+\eta, \\
\LL^\prime(\eta) & = -\frac{1}{2}(u_x^2+u_z^2)+\frac{u_y^2}{2(1+\eta)^2}(\eta_x^2+\eta_z^2) + \frac{u_y^2}{2(1+\eta)^2}-u_x\Bigg|_{y=1},
\end{align*}
which define analytic functions $\KK^\prime$, $\LL^\prime \colon U \to H^{1}(\R^2)$.
\end{lemma}

Writing
$$\KK(\eta) = \sum_{n=1}^\infty \KK_{2n}(\eta),
\qquad \eta \in U,$$
where $\KK_n(\eta)=\frac{1}{n!}\mathrm{d}^n\KK[0](\{\eta\}^n)$ (note that $\KK_{2n+1}(\eta)=0$ for each $n \in \N$),
we obtain the explicit formulae
$$
\KK_2(\eta)=\frac12 \int_{\R^2} (\eta^2+\beta \eta_x^2+\beta \eta_z^2)\dx\dz, \qquad
\KK_4(\eta)=-\frac{\beta}{8}\int_{\R^2} \left(\eta_x^2+\eta_z^2\right)^2\dx\dz
$$
and
$$
\KK_2^\prime(\eta)=\eta-\beta \eta_{xx} - \beta \eta_{zz}, \qquad
\KK_4^\prime(\eta)=\frac{\beta}{2} ((\eta_x^2+\eta_z^2)\eta_x)_x + \frac{\beta}{2}(\eta_x^2+\eta_z^2)\eta_z)_z.
$$
Semi-explicit formulae are also available for
the leading-order terms in the corresponding series representation
$$\LL(\eta) = \sum_{n=2}^\infty \LL_n(\eta)= \frac{1}{2}\sum_{n=2}^\infty \int_{\R^2} \eta K_n(\eta) \eta \dx \dz,\qquad \eta \in U,$$
where $K_n(\eta)=\frac{1}{n!}\mathrm{d}^n K[0](\{\eta\}^n)$, $\LL_n(\eta)=\frac{1}{n!}\mathrm{d}^n\LL[0](\{\eta\}^n)$
(see Buffoni \emph{et al.} \cite[Lemma 2.30 and Corollary 2.31]{BuffoniGrovesSunWahlen13}).

\begin{lemma} \label{lem:L2 to L4}
The functions $\LL_2$, $\LL_3$, $\LL_4: U \rightarrow \R$ and $\LL_2^\prime$, $\LL_3^\prime$, $\LL_4^\prime: U \rightarrow H^1(\R^2)$ are
given by the formulae
\begin{align*}
\LL_2(\eta)&=\frac{1}{2} \int_{\R^2} \eta K_0\eta \dx \dz, \\
\LL_3(\eta)&=
\frac{1}{2} \int_{\R^2} ( \eta_x^2 \eta 
-\eta (K_0 \eta)^2-\eta (L_0 \eta)^2)\dx \dz, \\
\LL_4(\eta)
&=\frac{1}{2}\int_{\R^2} \left(
K_0 (\eta K_0\eta)\, \eta K_0\eta +2 L_0(\eta L_0 \eta) \, \eta K_0\eta+\eta L_0 \eta \, H_0 (\eta L_0\eta) \right)\dx \dz\\
&\qquad\mbox{}+\frac{1}{2}\int_{\R^2}\eta^2 \left((K_0\eta)\eta_{xx}+(L_0\eta)\eta_{xz}\right) \dx\dz
\end{align*}
and
\begin{align*}
\LL_2^\prime(\eta)&=K_0 \eta, \\
\LL_3^\prime(\eta)&=\frac{1}{2}\left(\eta_x^2 - (K_0 \eta)^2 -(L_0 \eta)^2
-2(\eta_x\eta)_x-2K_0(\eta K_0\eta)-2L_0(\eta L_0\eta) \right), \\
\mathcal L_4^\prime(\eta)&=K_0\eta\, K_0(\eta K_0\eta) +K_0\eta \,L_0 (\eta L_0 \eta)
+L_0\eta \,L_0(\eta K_0\eta)+L_0\eta \, H_0(\eta L_0 \eta)\\
&\qquad\mbox{}+\eta ((K_0 \eta)\eta_{xx}+(L_0\eta)\eta_{xz}) +K_2(\eta)\eta,
\end{align*}
where
\[
\FF[K_0 \xi]=\frac{k_1^2}{|k|^2}f(|k|) \hat \xi, \quad \FF[L_0 \xi]=\frac{k_1 k_2}{|k|^2}f(|k|) \hat \xi, \quad \FF[H_0 \xi]=\frac{k_2^2}{|k|^2}f(|k|) \hat \xi.
\]
\end{lemma}

\subsubsection*{Weak continuity of the gradients}

\begin{lemma}
The function $\KK^\prime: U \rightarrow H^1(\R^2)$ is weakly continuous.
\end{lemma}
\begin{proof}
Because of the calculation
$$\langle \KK^\prime(\eta),\tilde{\eta} \rangle_1 = 
\langle \KK^\prime(\eta),\tilde{\eta} \rangle_0 - \langle \KK^\prime(\eta),\tilde{\eta}_{xx} \rangle_0 - \langle \KK^\prime(\eta),\tilde{\eta}_{zz} \rangle_0$$
for $\eta \in U$, $\tilde{\eta} \in C_0^\infty(\R^2)$  it suffices to show that $\KK^\prime$ is a weakly
continuous function $U \rightarrow L^2(\R^2)$.

Suppose that $\{\eta_n\} \subset U$ converges weakly in $H^3(\R^2)$ to $\eta_\infty \in U$, so that
$\{\eta_n\}$, $\{\eta_{nx}\}$, $\{\eta_{nz}\}$ converge strongly
in $C_\mathrm{loc}(\R^2)$ to $\eta_\infty$, $\eta_{\infty x}$, $\eta_{\infty z}$. Using the formula
$$\langle \KK^\prime(\eta),\tilde{\eta} \rangle _0 =\int_{\R^2}\left(
\frac{\beta\eta_x\tilde{\eta}_x}{\sqrt{1+\eta_x^2+\eta_z^2}}
+\frac{\beta\eta_z\tilde{\eta}_z}{\sqrt{1+\eta_x^2+\eta_z^2}}+\eta\tilde{\eta}\right)\dx\dz,$$
we conclude that $\langle \KK^\prime(\eta_n),\tilde{\eta} \rangle_0 \rightarrow \langle \KK^\prime(\eta_\infty),\tilde{\eta} \rangle_0$
for each $\tilde{\eta} \in C_0^\infty(\R^2)$.
\end{proof}

To obtain the corresponding result for $\LL^\prime$ we first establish some 
further mapping properties of the operator
$K(\cdot)\colon U \to \LL(H^{5/2}(\R^2), H^{3/2}(\R^2))$.
For this purpose we note that the solution of the boundary-value problem
\eqref{flatKproblem 1}--\eqref{flatKproblem 3} (with $\eta \in U$, $\xi \in H^{5/2}(\R^2)$)
can be characterised as the unique solution of the equation
$$u=\Gamma(F_1(\eta,u),F_2(\eta,u),F_3(\eta,u),\xi)$$
(see Proposition \ref{prop:Defn of Gamma}).

\begin{proposition} \label{prop:Properties of K}
Suppose that $\{\eta_n\} \subset U$ converges weakly in $H^3(\R^2)$
to $\eta_\infty \in U$ and $\{\xi_n\} \subset H^{5/2}(\R^2)$ converges weakly
in $H^{5/2}(\R^2)$ to $\xi_\infty$.
\begin{itemize}
\item[(i)]
The sequence $\{K(\eta_{n})\xi_n\}$ converges weakly in
$H^{3/2}(\R^2)$ to $K(\eta_{\infty})\xi_\infty$.
\item[(ii)]
The sequence
$\{(K(\eta_{\infty})-K(0))\xi_n\}$ converges strongly in
$H^{1/2}(\R^2)$ to
$(K(\eta_{\infty})-K(0))\xi_\infty$.
\item[(iii)]
For each $\rho \in H^3(\R^2)$ the sequence
$\left\{\big(\mathrm{d}K[\eta_n](\rho)\big)\xi_n\right\}$
converges strongly in
$H^{1/2}(\R^2)$ to $\big(\mathrm{d}K[\eta_\infty](\rho)\big)\xi_\infty$.
\end{itemize}
\end{proposition}

\begin{proof}
(i) Let $u_n$ and $u_\infty$ be the solutions to \eqref{flatKproblem 1}--\eqref{flatKproblem 3} with $\eta$, $\xi$
replaced by respectively $\eta_n$, $\xi_n$ and $\eta_\infty$, $\xi_\infty$, so that
$$u_n=\Gamma(F_1(\eta_n,u_n),F_2(\eta_n,u_n),F_3(\eta_n,u_n),\xi_n)$$
and
$$u_\infty=\Gamma(F_1(\eta_\infty,u_\infty),F_2(\eta_\infty,u_\infty),F_3(\eta_\infty,u_\infty),\xi_\infty).$$
Since $\{\eta_n\}$ and $\{\xi_n\}$ are bounded in
respectively $H^3(\R^2)$ and $H^{5/2}(\R^2)$, it follows from Lemma \ref{lem:u analytic} that $\{u_n\}$ is bounded
in $H^3_\star(\Sigma)$. The following argument shows that any weakly convergent subsequence of $\{u_n\}$ has weak limit $u_\infty$,
so that $\{u_n\}$ itself converges weakly to $u_\infty$;
in particular $K(\eta_{n})\xi_n=-u_{nx}(x,1,z) \rightharpoonup -u_{\infty x}(x,1,z)=K(\eta_{\infty})\xi_\infty$ in $H^{3/2}(\R^2)$.

Suppose that (a subsequence of) $\{u_n\}$ converges weakly to $u_0$ in $H^3_\star(\Sigma)$.
Observing that $\{\eta_n\}$ converges strongly  in $H^2_\mathrm{loc}(\R^2)$ to $\eta_\infty$,
$\{u_n\}$ converges strongly in $H^2_{\star,\mathrm{loc}}(\Sigma)$ to $u_0$ and hence
$\{F_j(\eta_n,u_n)\}$ converges strongly in $H^1_\mathrm{loc}(\Sigma)$ to
$F_j(\eta_\infty,u_0)$, $j=1,2,3$, we find that
$u_0$ is the solution to \eqref{flatKproblem 1}--\eqref{flatKproblem 3} with $\eta$ and $\xi$
replaced by respectively $\eta_\infty$ and $\xi_\infty$, so that $u_0=u_\infty$.

(ii) Define
$$v_n = \Gamma(F_1(\eta_\infty,u_n),F_2(\eta_\infty,u_n),F_3(\eta_\infty,u_n),0)$$
and repeat the argument used in part (i): the sequence $\{v_n\}$ converges weakly in $H^3_\star(\Sigma)$
to $v_\infty$, and
$$v_\infty = \Gamma(F_1(\eta_\infty,u_\infty),F_2(\eta_\infty,u_\infty),F_3(\eta_\infty,u_\infty),0).$$
Furthermore $\{F_j(\eta_\infty,u_n)\}$ converges strongly to $F_j(\eta_\infty,u_\infty)$ in $H^1(\Sigma)$
since
\begin{align*}
\|F_j(\eta_\infty,u_n-u_\infty)\|_1 & \leq \|F_j(\eta_\infty,u_n-u_\infty)\|_{H^1(|(x,z)|<R)} + \|F_j(\eta_\infty,u_n-u_\infty)\|_{H^1(|(x,z)|>R)} \\
& \lesssim \|F_j(\eta_\infty,u_n-u_\infty)\|_{H^1(|(x,z)|<R)} + \|\eta_\infty\|_{H^3(|(x,z)|>R)} \\
& \rightarrow 0
\end{align*}
as $n \rightarrow \infty$ (note that $\|\eta_\infty\|_{H^3(|(x,z)|>R)} \rightarrow 0$ as $R \rightarrow \infty$ and
$\{F_j(\eta_\infty,u_n)\}$ converges strongly in $H^1(|(x,z)|<R)$ to $F_j(\eta_\infty,u_\infty)$). It follows that
$\{v_n\}$ converges strongly in $H_\star^2(\Sigma)$ to $v_\infty$ and $(K(\eta_\infty)-K(0))\xi_n=-v_{nx}(x,1,z)
\rightarrow -v_{\infty x}(x,1,z) = (K(\eta_\infty)-K(0))\xi_\infty$ in $H^{1/2}(\R^2)$.

(iii) Let $w_n=\mathrm{d}_1u_n[\eta_n,\xi_n](\rho)$, so that
\begin{align}
w_n&=\Gamma(F_1(\eta_n,w_n),F_2(\eta_n,w_n),F_3(\eta_n,w_n),0) \nonumber \\
& \qquad\mbox{}
+\Gamma(
\mathrm{d}_1F_1[\eta_n,u_n](\rho),
\mathrm{d}_1F_2[\eta_n,u_n](\rho),
\mathrm{d}_1F_3[\eta_n,u_n](\rho),0), \label{Abstract eqn}
\end{align}
where
\begin{align*}
\mathrm{d}_1F_1[\eta,u](\rho)&=-\rho u_x+y\rho_x u_y, \\
\mathrm{d}_1F_2[\eta,u](\rho)&=-\rho u_z+y\rho_z u_y, \\
\mathrm{d}_1F_3[\eta,u](\rho)&=
\frac{\rho u_y}{1+\eta}
-\frac{\eta u_y\rho}{(1+\eta)^2}
+y\rho_x u_x+
y\rho_z u_z \\
& \qquad\mbox{}
+\frac{y^2\rho}{(1+\eta)^2}
(\eta_x^2+\eta_z^2) u_y
-\frac{2y^2}{1+\eta}
(\eta_x\rho_x+\eta_z\rho_z) u_y;
\end{align*}
the usual argument shows that $\{w_n\}$ converges weakly in $H^3_\star(\Sigma)$ to $w_\infty$ and
that\linebreak $w_\infty=\mathrm{d}_1u_\infty[\eta_\infty,\xi_\infty](\rho)$.

The argument used in part (ii) shows that
$\{\mathrm{d}_1F_j[\eta_n,u_n](\rho)\}$ converges strongly
in $H^1_\star(\Sigma)$ to $\mathrm{d}_1F_j[\eta_\infty,u_\infty](\rho)$,
so that $\{c_n\}$ with $c_n = \Gamma(
\mathrm{d}_1F_1[\eta_n,u_n](\rho),
\mathrm{d}_1F_2[\eta_n,u_n](\rho),
\mathrm{d}_1F_3[\eta_n,u_n](\rho),0)$
is strongly convergent in $H_\star^2(\Sigma)$. Define
$A(\eta) \in \LL(H_\star^2(\Sigma))$ by the formula
$A(\eta) =\linebreak \Gamma(F_1(\eta,\cdot),F_2(\eta,\cdot),F_3(\eta,\cdot),0)$,
choosing $M$ small enough so that $\sup_{\eta \in U}\|A(\eta)\| < 1$,
and observe that $A_n=A(\eta_n)$
has the property that $\{A_n b_n\}$ is strongly convergent
whenever $\{b_n\}$ is strongly convergent.
Writing \eqref{Abstract eqn} in the abstract form
$$w_n = A_n w_n + c_n,$$
we find from a familiar Neumann-series argument that $\{w_n\}$ is a Cauchy sequence
in $H^2_\star(\Sigma)$ which therefore converges strongly to its weak limit $w_\infty$.
It follows that $\big(\mathrm{d}K[\eta_n](\rho)\big)\xi_n=-w_{n x}(x,1,z)
\rightarrow -w_{\infty x}(x,1,z) = \big(\mathrm{d}K[\eta_\infty](\rho)\big)\xi_\infty$
in $H^{1/2}(\R^2)$.
\end{proof}
\begin{lemma}
The function $\LL^\prime: U \rightarrow H^1(\R^2)$ is weakly continuous.
\end{lemma}
\begin{proof}
It suffices to show that $\LL^\prime$ is a weakly continuous function $U \rightarrow L^2(\R^2)$.
Suppose that $\{\eta_n\} \subset U$ converges weakly in $H^3(\R^2)$ to $\eta_\infty \in U$. Using the formula
$$\langle \LL^\prime(\eta),\rho \rangle_0 = \langle K(\eta)\eta, \rho \rangle_0
+ \frac{1}{2}\int_{\R^2} \eta \big(\mathrm{d}K[\eta](\rho)\big)\eta \dx \dz$$
and Proposition \ref{prop:Properties of K}, we find that
$\langle \LL^\prime(\eta_n),\rho \rangle_0 \rightarrow \langle \LL^\prime(\eta_\infty),\rho \rangle_0$
for each $\rho \in C_0^\infty(\R^2)$.
\end{proof}

\begin{corollary} \label{cor:J' is wc}
The function $\JJ_\varepsilon^\prime: U \rightarrow H^1(\R^2)$ is weakly continuous.
\end{corollary}

\begin{remark} \label{rem:Gradients wc}
The functions $\LL_3^\prime$, $\LL_4^\prime$ and $\KK_4^\prime: U \rightarrow H^1(\R^2)$ are also weakly continuous.

Suppose that $\{\eta_n\} \subset U$ converges weakly in $H^3(\R^2)$ to $\eta_\infty \in U$.
It follows that $\{K_0\eta_n\}$, $\{L_0\eta_n\}$ converge strongly in $L^4_\mathrm{loc}(\R^2)$ to $K_0\eta_\infty$, $L_0\eta_\infty$
and $\{\eta_n K_0\eta_n\}$, $\{\eta_n L_0\eta_n\}$,
$\{(K_0 \eta_n)^2\}$, $\{(L_0 \eta_n)^2\}$, $\{\eta_{nx}^2\}$ converge strongly  in $L^2_\mathrm{loc}(\R^2)$ to
$\eta_\infty K_0\eta_\infty$, $\eta_\infty L_0\eta_\infty$, $(K_0\eta_\infty)^2$, $(L_0\eta_\infty)^2$, $\eta_{\infty x}^2$. Examining the expression for $\LL_3^\prime(\eta)$ given in Lemma
\ref{lem:L2 to L4}, we conclude that $\langle \LL_3^\prime(\eta_n), \tilde{\eta} \rangle_0$ converges to
$\langle \LL_3^\prime(\eta_\infty), \tilde{\eta} \rangle_0$ for each $\tilde{\eta} \in C_0^\infty(\R^2)$.

Similar arguments show that $\langle \KK_4^\prime(\eta_n), \tilde{\eta} \rangle_0 \rightarrow \langle \KK_4^\prime(\eta),\tilde{\eta} \rangle_0$
and $\langle \LL_4^\prime(\eta_n), \tilde{\eta} \rangle_0 \rightarrow \langle \LL_4^\prime(\eta),\tilde{\eta} \rangle_0$ as $n \rightarrow \infty$
for each $\tilde{\eta}\in C_0^\infty(\R^2)$.
\end{remark}

\subsubsection*{Further notation}

Finally, we denote the superquadratic part of $\JJ_\varepsilon(\eta)$ by $\NN(\eta)$, that is write
\begin{align*}
\NN(\eta):=&\,\JJ_\varepsilon(\eta)-(\KK_2(\eta)-\Lambda(1-\varepsilon^2)\LL_2(\eta))\\
=&\, \KK_\mathrm{nl}(\eta) - \Lambda(1-\varepsilon^2)\LL_\mathrm{nl}(\eta),
\end{align*}
where
$$\KK_\mathrm{nl}(\eta)= \sum_{n=2}^\infty \KK_{2n}(\eta), \qquad \LL_\mathrm{nl}(\eta)=\sum_{n=3}^\infty \LL_n(\eta);$$
in view of the above calculations we also use the notation
$$\KK_\mathrm{r}(\eta):= \KK_\mathrm{nl}(\eta)-\KK_4(\eta), \qquad
\LL_\mathrm{r}(\eta):= \LL_\mathrm{nl}(\eta)-\LL_3(\eta)-\LL_4(\eta).$$
Note that $\NN$, $\KK_\mathrm{nl}$, $\LL_\mathrm{nl}$, $\KK_\mathrm{r}$ and $\LL_\mathrm{r}:U \rightarrow H^1(\R^2)$
are also weakly continuous.

\subsection{The reduction procedure}
\label{section: Reduction}

The next step is to decompose $\XX=H^3(\R^2)$ into the direct sum of the weakly closed subspaces
$\XX_1 = \chi(D)\XX$ and $\XX_2 = (1-\chi(D))\XX$. Observe that
$\eta \in U$ is a critical point of $\JJ_\varepsilon$, that is
$$
\mathrm{d}\JJ_\varepsilon[\eta](\rho)=0
$$
for all $\rho \in \XX$, if and only if
$$\mathrm{d}\JJ_\varepsilon[\eta_1+\eta_2](\rho_1)=0, \qquad \mathrm{d}\JJ_\varepsilon[\eta_1+\eta_2](\rho_2)=0,$$
or equivalently
\begin{equation}
\label{Two grad eqns}
\langle \JJ_\varepsilon^\prime(\eta_1+\eta_2), \rho_1\rangle_0=0, \qquad \langle \JJ_\varepsilon^\prime(\eta_1+\eta_2), \rho_2\rangle_0=0,
\end{equation}
for all $\rho_1 \in \XX_1$ and $\rho_2 \in \XX_2$. Equations \eqref{Two grad eqns} are given explicitly by
\begin{align}
\eta_1-\beta\eta_{1xx}-\beta\eta_{1zz}-\Lambda K_0\eta_1 +\varepsilon^2\Lambda K_0\eta_1 \hspace{1.8in}\nonumber\\
\mbox{}+\chi(D)\left(\NN^\prime(\eta_1+\eta_2)+\Lambda(1-\varepsilon^2)\LL_3^\prime(\eta_1)\right)&=0, \label{Euler-Lagrange X_1 component}  \\
\underbrace{\eta_2-\beta\eta_{2xx}-\beta\eta_{2zz}-\Lambda K_0\eta_2 }_{\displaystyle = \tilde{g}(D)\eta_2}+\varepsilon^2\Lambda K_0\eta_2+(1-\chi(D))\NN^\prime(\eta_1+\eta_2)&=0,\label{Euler-Lagrange X_2 component}
\end{align}
where $$\tilde{g}(k):=g(|k|) +\Lambda \frac{k_2^2}{|k|^2}f(|k|) \geq 0$$
with equality if and only if $k=\pm (\omega,0)$ (see the comments to equation \eqref{Defn of g})
and we have used the fact that $\chi(D)\LL_3^\prime(\eta_1)$ vanishes (so that the nonlinear term in
\eqref{Euler-Lagrange X_1 component} is at leading order cubic in $\eta_1$). We accordingly write
$$\eta_2 = F(\eta_1)+\eta_3, \qquad F(\eta_1):=
\Lambda(1-\varepsilon^2)\FF^{-1}\left[\frac{1-\chi(k)}{\tilde g(k)}\FF[\LL_3^\prime(\eta_1)]\right]$$
and \eqref{Euler-Lagrange X_2 component}
in the form
\begin{equation}
\eta_3 = -\FF^{-1}\left[\frac{1-\chi(k)}{\tilde g(k)}\FF\left[\Lambda(1-\varepsilon^2)\LL_3^\prime(\eta_1)+\NN^\prime(\eta_1+F(\eta_1)+\eta_3)+\Lambda\varepsilon^2 K_0(F(\eta_1)+\eta_3)\right]\right] \label{eta3 eqn}
\end{equation}
(with the requirement that $\eta_1+F(\eta_1)+\eta_3 \in U$).

\begin{proposition}
\label{prop:bounded mapping}
The mapping $$f\mapsto \FF^{-1}\left[\frac{1-\chi(k)}{\tilde g(k)}\hat f\right]$$ defines a bounded linear operator $H^1(\R^2)\to H^3(\R^2)$.
\end{proposition}

We proceed by solving \eqref{eta3 eqn} for $\eta_3$ as a function of $\eta_1$ using the following fixed-point theorem, which is a straightforward extension of a standard result in nonlinear analysis.

\begin{theorem}
\label{thm:fixed-point}
Let $\YY_1$, $\YY_2$ be Banach spaces, $Y_1$, $Y_2$ be closed sets in, respectively, $\YY_1$, $\YY_2$ containing the origin and $G\colon Y_1\times Y_2 \to \YY_2$ be a smooth function. Suppose that there exists a 
continuous function $r\colon Y_1\to [0,\infty)$ such that
$$
\|G(y_1,0)\|\le \frac{r}{2}, \quad \|\mathrm{d}_2 G[y_1,y_2]\|\le \frac13
$$ 
for each $y_2\in \bar B_r(0)\subset Y_2$ and each $y_1\in Y_1$.

Under these hypotheses there exists for each $y_1\in Y_1$ a unique solution $y_2=y_2(y_1)$ of the fixed-point equation
$$
y_2=G(y_1,y_2)
$$
satisfying $y_2(y_1)\in \bar B_r(0)$. Moreover $y_2(y_1)$ is a smooth function of $y_1\in Y_1$ and in particular satisfies the estimate
$$
\|\mathrm{d} y_2[y_1]\|\le 2\|\mathrm{d}_1 G[y_1, y_2(y_1)]\|
$$
for its first derivative
and the estimate
\begin{align*}
\|\mathrm{d}^2 y_2[y_1]& \|\le 2\|\mathrm{d}_1^2G[y_1,y_2(y_1)]\| \\
& \qquad\mbox{}+8\|\mathrm{d}_1\mathrm{d}_2 G[y_1,y_2(y_1)]\| \|\mathrm{d}_1 G[y_1, y_2(y_1)]\|+8\|\mathrm{d}_2^2 G[y_1,y_2(y_1)]\|
\|\mathrm{d}_1 G[y_1, y_2(y_1)]\|^2
\end{align*}
for its second derivative.
\end{theorem}

We apply Theorem \ref{thm:fixed-point} to equation \eqref{eta3 eqn} with
$\YY_1=\XX_1$, $\YY_2=\XX_2$, equipping
$\XX_1$ with the scaled norm
$$
\nn \eta \nn:=\left( \int_{\R^2} (1+\varepsilon^{-2}((|k_1|-\omega)^2+k_2^2))|\hat \eta(k)|^2\dk_1 \dk_2\right)^{\!1/2}
$$
and $\XX_2$ with the usual norm for $H^3(\R^2)$, and taking $Y_1=X_1$, $Y_2=X_3$, where
$$
X_1=\{\eta_1\in \XX_1 \colon \nn \eta_1\nn \le R_1\}, \qquad
X_3=\{\eta_3\in \XX_2 \colon \| \eta_3\|_3 \le R_3\};
$$
the function $G$ is given by the right-hand side of \eqref{eta3 eqn}. The following proposition shows that
\begin{equation}
\|\hat{\eta}_1\|_{L^1(\R^2)} \lesssim \varepsilon^\theta \nn \eta_1 \nn, \qquad \eta_1 \in \XX_1,
\label{Triple norm large}
\end{equation}
for each fixed $\theta \in (0,1)$, so that we can guarantee that $\|\hat{\eta}_1\|_{L^1(\R^2)} < M/2$ for all
$\eta_1 \in X_1$ for an arbitrarily large value
of $R_1$; the value of $R_3$
is then constrained by the requirement that $\|F(\eta_1) + \eta_3\|_3 < M/2$ for all $\eta_1 \in X_1$ and $\eta_3 \in X_3$,
so that $\eta_1+F(\eta_1)+\eta_3 \in U=B_M(0)$
(Corollary \ref{cor:F estimates} below asserts that $\|F(\eta_1)\|_3 = O(\varepsilon^\theta)$ uniformly over $\eta_1 \in X_1$).

\begin{proposition}
\label{sup estimate}
The estimate 
\[
\|\hat \eta_1\|_{L^1(\R^2)} \lesssim
\varepsilon |\! \log \varepsilon |  \nn \eta_1\nn
\]
holds for each $\eta_1\in \XX_1$.
\end{proposition}
\begin{proof}
Observe that
\begin{align*}
\int_{\R^2} |\hat \eta_1(k)|\dk_1\dk_2
&= \int_{\R^2}\frac{(1+\varepsilon^{-2}((|k_1|-\omega)^2+k_2^2))^{1/2}}{(1+\varepsilon^{-2}((|k_1|-\omega)^2+k_2^2))^{1/2}} |\hat \eta_1(k)| \dk_1 \dk_2\\
&\le 2\nn \eta\nn \left( \int_{B_\delta(\omega,0)} \frac{1}{1+\varepsilon^{-2}((k_1-\omega)^2+k_2^2)} \dk_1 \dk_2 \right)^{1/2}, \\
&=2\varepsilon\nn \eta\nn  \left(\int_{|k|<\delta/\varepsilon} \frac{1}{1+|k|^2}  \dk_1 \dk_2\right)^{1/2} \\
& = 2\sqrt{\pi}\varepsilon (\log(1+\delta^2\varepsilon^{-2}))^{1/2} \nn \eta\nn. \qedhere
\end{align*}
\end{proof}

We proceed by systematically estimating each term
appearing in the equation for $G$, writing
$$
\Lambda(1-\varepsilon^2)\LL_3^\prime(\eta_1)+\NN^\prime(\eta) \\
= \KK_\mathrm{nl}^\prime(\eta_1+F(\eta_1)+\eta_3) - \Lambda(1-\varepsilon^2)\big( \LL_3^\prime(\eta) - \LL_3^\prime(\eta_1) + \LL_4^\prime(\eta)
+\LL_\mathrm{r}^\prime(\eta)\big),
$$
where $\eta=\eta_1+F(\eta_1)+\eta_3$, and using the inequalities \eqref{Triple norm large} and
$$\|\eta_1\|_3 \lesssim \nn \eta_1 \nn$$
to handle $\eta_1$; note in particular that
$$\|\eta\|_\ZZ \lesssim \varepsilon^\theta \nn \eta_1 \nn + \|\eta_3\|_3, \qquad
\|\eta\|_3 \lesssim \nn \eta_1 \nn + \|\eta_3\|_3$$
for each $\eta \in H^3(\R^2)$.

In order to estimate $F(\eta_1)$ we write $\LL_3^\prime(\eta)=m(\{\eta\}^2)$, where
\begin{align*}
m(u,v)&=\frac{1}{2}\left(u_x v_x - (K_0 u) (K_0 v) -(L_0 u) (L_0 v) \right)\\
&\qquad\mbox{}+\frac{1}{2}\left(
-(u_xv +u v_x)_x-K_0(u K_0 v+v K_0 u)-L_0(u L_0v +vL_0u)
\right)
\end{align*}
(see Lemma \ref{lem:L2 to L4}), and note that
$$\mathrm{d}\LL_3^\prime[\eta](v)=2m(\eta,v), \qquad \mathrm{d^2}\LL_3^\prime[\eta](v,w)=2m(v,w).$$

\begin{proposition}
\label{prop:m estimate}
The estimate
$$\|m(u,v)\|_1 \lesssim \|u\|_\ZZ \|v\|_3,$$
holds for each $u$, $v\in H^3(\R^2)$.
\end{proposition}
\begin{proof}
We estimate
\begin{align*}
\|m(u,v)\|_1 & \lesssim (\|u_1\|_{3,\infty}+\|K_0u_1\|_{2,\infty} + \|L_0 u_1\|_{2,\infty}+\|u_2\|_3)\|v\|_3 \\
& \lesssim (\|\hat{u}_1\|_{L^1(\R^2)}+\|u_2\|_3) \|v\|_3 \\
& = \|u\|_\ZZ \|v\|_3,
\end{align*}
where the second line follows from the fact that
$$\|u_1\|_{m,\infty},\ \|K_0 u_1\|_{m,\infty},\ \|L_0 u_1\|_{m,\infty}, \|H_0 u_1\|_{m,\infty} \lesssim \|\hat{u}_1\|_{L_1(\R^2)}$$
for each $m \in {\mathbb N}_0$ (since $\hat{u}_1$ has compact support).
\end{proof}

\begin{corollary}
\label{cor:F estimates}
The estimates
$$\|F(\eta_1)\|_3\lesssim \varepsilon^\theta \nn \eta_1\nn^2,\quad
\|\mathrm{d}F[\eta_1]\|_{\LL(\XX_1, \XX_2)}  \lesssim \varepsilon^\theta \nn \eta_1\nn,\quad
\|\mathrm{d}^2F[\eta_1]\|_{\LL^2(\XX_1,\XX_2)}  \lesssim \varepsilon^\theta$$
hold for each $\eta_1\in X_1$, where $\LL(\XX_1,\XX_2)$ and $\LL^2(\XX_1,\XX_2)$ denote the spaces
of bounded linear and bilinear operators $\XX_1 \rightarrow \XX_2$.
\end{corollary}

\begin{remark} \label{rem:K0Feta1 remark}
Noting that
$$K_0 F(\eta_1) = \Lambda(1-\varepsilon^2)\FF^{-1}\left[\frac{1-\chi(k)}{\tilde g(k)} \frac{k_1^2}{|k|^2}f(|k|) \LL_3^\prime(\eta_1)\right]$$
and that $\LL_3^\prime(\eta_1)$ has compact support, one finds
that $K_0 F(\eta_1)$ satisfies the same estimates as $F(\eta_1)$.
\end{remark}

The quantity 
\begin{align*}
\AA(\eta_1,\eta_3):&=\LL_3^\prime(\eta_1+F(\eta_1)+\eta_3)-\LL_3^\prime(\eta_1)\\
&= 2m(\eta_1, F(\eta_1)+\eta_3)+m(F(\eta_1)+\eta_3,F(\eta_1)+\eta_3)
\end{align*}
is estimated by combining Proposition \ref{prop:m estimate} and Corollary \ref{cor:F estimates} using the chain rule.
\begin{lemma}
\label{lem:3 estimates}
The estimates
\begin{list}{(\roman{count})}{\usecounter{count}}
\item
$\|\AA(\eta_1,\eta_3)\|_1\lesssim \varepsilon^{2\theta} \nn \eta_1\nn^3+\varepsilon^\theta \nn \eta_1\nn^2\|\eta_3\|_3
+\varepsilon^\theta \nn \eta_1\nn\|\eta_3\|_3+\|\eta_3\|_3^2$,
\item
$\|\mathrm{d}_1\AA[\eta_1,\eta_3]\|_{\LL(\XX_1,H^1(\R^2))}\lesssim \varepsilon^{2\theta} \nn \eta_1\nn^2
+\varepsilon^\theta \nn \eta_1\nn\|\eta_3\|_3+\varepsilon^\theta \|\eta_3\|_3$,
\item
$\|\mathrm{d}_2\AA[\eta_1,\eta_3]\|_{\LL(\XX_2,H^1(\R^2))}\lesssim \varepsilon^\theta \nn \eta_1\nn+\|\eta_3\|_3$,
\item
$\|\mathrm{d}_1^2\AA[\eta_1,\eta_3]\|_{\LL^2(\XX_1,H^1(\R^2))}\lesssim \varepsilon^{2\theta} \nn \eta_1\nn+\varepsilon^\theta \|\eta_3\|_3$,
\item
$\|\mathrm{d}_1\mathrm{d}_2\AA[\eta_1,\eta_3]\|_{\LL^2(\XX_1 \times \XX_2,H^1(\R^2))}\lesssim\varepsilon^\theta$,
\item
$\|\mathrm{d}_2^2\AA[\eta_1,\eta_3]\|_{\LL^2(\XX_2,H^1(\R^2))}\lesssim 1$
\end{list}
hold for each $\eta_1\in X_1$ and $\eta_3\in X_3$, where $\LL(\XX_1,H^1(\R^2))$, $\LL(\XX_2,H^1(\R^2))$ and
$\LL^2(\XX_1,H^1(\R^2))$, $\LL^2(\XX_1 \times \XX_2,H^1(\R^2))$, $\LL^2(\XX_2,H^1(\R^2))$ denote the Banach spaces
of bounded linear and bilinear operators from the indicated spaces to $H^1(\R^2)$.
\end{lemma}

The quantity $\LL_4^\prime(\eta_1+F(\eta_1)+\eta_3)$
is estimated by writing
$$\LL_4^\prime(\eta) = n_\mathrm{sym}(\{\eta\}^3) + H(\eta),$$
where
$$
n_\mathrm{sym}(u_1,u_2,u_3)=\frac{1}{6}\sum_{\sigma \in S_3} n(u_{\sigma(1)},u_{\sigma(2)}, u_{\sigma(3)}), \qquad H(\eta)=K_2(\eta)\eta
$$
and
\begin{align*}
n(u,v,w)&=K_0u\, K_0(v K_0w) +K_0u \,L_0 (v L_0 w)
+L_0u \,L_0(v K_0w)+L_0u \, H_0(v L_0 w)\\
&\qquad\mbox{}+u K_0 v\, w_{xx}+u L_0v\,w_{xz}
\end{align*}
(see Lemma \ref{lem:L2 to L4}).

\begin{lemma} \label{lem:4 estimates}
The estimates
\begin{align*}
\|\LL_4^\prime(\eta)\|_1 & \lesssim \|\eta\|_\ZZ^2 \|\eta\|_3, \\
\|\mathrm{d}\LL_4^\prime[\eta](v)\|_1 & \lesssim \|\eta\|_\ZZ^2 \|v\|_3 + \|\eta\|_\ZZ \|\eta\|_3 \|v\|_\ZZ, \\
\|\mathrm{d}^2\LL_4^\prime[\eta](v,w)\|_1 & \lesssim \|\eta\|_3 \|v\|_\ZZ\|w\|_\ZZ + \|\eta\|_\ZZ \|v\|_\ZZ \|w\|_3 + \|\eta\|_\ZZ\|v\|_3 \|w\|_\ZZ \end{align*}
hold for each $\eta \in U$ and $v,w \in H^3(\R^2)$.
\end{lemma}
\begin{proof}
Using the estimate
$$\|fg\|_1 \lesssim \|f\|_2 \|g\|_1$$
(see H\"{o}rmander \cite[Theorem 8.3.1]{Hoermander}), one finds that
\begin{align*}
\| K_0u K_0 (v K_0 w) \|_1 & \lesssim (\|K_0 u_1 \|_{1,\infty} + \|K_0u_2\|_2)\|K_0(v K_0 w)\|_1 \\
& \lesssim (\|K_0 u_1 \|_{1,\infty} + \|K_0u_2\|_2)\|v K_0 w\|_2 \\
& \lesssim (\|K_0 u_1 \|_{1,\infty} + \|K_0u_2\|_2)(\|v_1 \|_{2,\infty} + \|v_2\|_2)\|K_0 w\|_2 \\
& \lesssim (\|\hat{u}_1\|_{L^1(\R^2)} + \|u_2\|_3)(\|\hat{v}_1\|_{L^1(\R^2)} + \|v_2\|_3)\|w\|_3 \\
& = \|u\|_\ZZ\|v\|_\ZZ\|w\|_3, \\
\\
\| K_0v K_0 (w K_0 u) \|_1 & \lesssim (\|K_0 v_1 \|_{1,\infty} + \|K_0v_2\|_2)\|w K_0 u\|_2 \\
& \lesssim (\|K_0 v_1 \|_{1,\infty} + \|K_0v_2\|_2) \|w\|_2 (\| K_0 u_1\|_{2,\infty}+\|K_0 u_2 \|_2) \\
& \lesssim (\|K_0 v_1 \|_{1,\infty} + \|K_0v_2\|_2) \|w\|_3 (\| K_0 u_1\|_{2,\infty}+\|u_2 \|_3) \\
& \lesssim \|u\|_\ZZ\|v\|_\ZZ\|w\|_3, \\
\\
\| K_0w K_0 (u K_0 v) \|_1 & \lesssim \|K_0 w \|_1 \|K_0(u_1 K_0 v_1)\|_{1,\infty} + \|K_0w\|_2(\|u_1K_0v_2\|_2 + \|u_2 K_0v\|_2) \\
& \lesssim \|w \|_3 \big( \|\FF[u_1 K_0 v_1] \|_{L^1(\R^2)} +\|u_1\|_{2,\infty} \|v_2\|_3
+\|u_2\|_3 (\| K_0 v_1\|_{2,\infty}+\|v_2 \|_3 )\big)\\
& \lesssim \|w \|_3 ( \|\hat{u}_1\|_{L^1(\R^2)} \|\FF[K_0 v_1] \|_{L^1(\R^2)} + \|u\|_\ZZ \|v\|_\ZZ) \\
& \lesssim \|w \|_3 ( \|\hat{u}_1\|_{L^1(\R^2)} \|\hat{v}_1\|_{L^1(\R^2)} + \|u\|_\ZZ \|v\|_\ZZ) \\
& \lesssim \|u\|_\ZZ\|v\|_\ZZ\|w\|_3
\end{align*}
for each $u$, $v$, $w \in H^3(\R^2)$
and the same estimates hold when any occurrence of $K_0$ is replaced by $L_0$ or $H_0$;
similar calculations show that
$$
\|u K_0 v\, w_{xx}\|_1,\  \|v K_0 w\, u_{xx}\|_1,\ \|w K_0 u\, v_{xx}\|_1 \lesssim \|u\|_\ZZ\|v\|_\ZZ\|w\|_3,
$$
and the same estimates hold for $u L_0 v\, w_{xz}$, $v L_0 w\, u_{xz}$ and $w L_0 u\, v_{xz}$.

Altogether the above calculations show that
$$\|n_\mathrm{sym}(u,v,w)\|_1 \lesssim \|u\|_\ZZ \|v\|_\ZZ \|w\|_3$$
for each $u$, $v$, $w \in H^3(\R^2)$; the lemma follows from this estimate and the inequalities
\begin{align*}
\|H(\eta)\|_1 & \lesssim \|\eta\|_\ZZ^2 \|\eta\|_3, \\
\|\mathrm{d}H[\eta](v)\|_1 & \lesssim \|\eta\|_\ZZ^2 \|v\|_3 + \|\eta\|_\ZZ \|\eta\|_3 \|v\|_\ZZ, \\
\|\mathrm{d}^2H[\eta](v,w)\|_1 & \lesssim \|\eta\|_3 \|v\|_\ZZ\|w\|_\ZZ + \|\eta\|_\ZZ \|v\|_\ZZ \|w\|_3 + \|\eta\|_\ZZ\|v\|_3 \|w\|_\ZZ \end{align*}
for $\eta \in U$ and $v$, $w \in H^3(\R^2)$ (see Corollary \ref{cor:K analytic}).
\end{proof}

The quantity
$\LL_\mathrm{r}^\prime(\eta_1+F(\eta_1)+\eta_3)$
is handled using the next lemma, which follows from Lemmata \ref{lem:gradients} and \ref{lem:u analytic}.

\begin{lemma}
The estimates
\begin{align*}
\|\LL_\mathrm{r}^\prime(\eta)\|_1 & \lesssim \|\eta\|_{\ZZ}^2 \|\eta\|_3^2, \\
\|\mathrm{d}\LL_\mathrm{r}^\prime[\eta](v)\|_1 & \lesssim \|\eta\|_{\ZZ}^2 \|\eta\|_3 \|v\|_3 + \|\eta\|_{\ZZ}  \|\eta\|_3^2\|v\|_\ZZ, \\
\|\mathrm{d}^2\LL_\mathrm{r}^\prime[\eta](v,w)\|_1 & \lesssim \|\eta\|_{\ZZ}^2 \|v\|_3 \|w\|_3 + \|\eta\|_3^2 \|v\|_\ZZ \|w\|_\ZZ\\
& \qquad\mbox{}+\|\eta\|_{\ZZ} \|\eta\|_3 \|v\|_3 \|w\|_\ZZ + \|\eta\|_{\ZZ}\|\eta\|_3 \|v\|_\ZZ \|w\|_3
\end{align*}
hold for each $\eta \in U$ and $v$, $w \in H^3(\R^2)$.
\end{lemma}

Finally, we examine the quantity $\KK_\mathrm{nl}^\prime(\eta_1+F(\eta_1)+\eta_3)$.

\begin{lemma} \label{lem:KKnlprime estimates}
The estimates
\begin{align*}
\|\KK_\mathrm{nl}^\prime(\eta)\|_1 & \lesssim \|\eta\|_\ZZ^2 \|\eta\|_3, \\
\|\mathrm{d}\KK_\mathrm{nl}^\prime[\eta](v)\|_1 & \lesssim \|\eta\|_\ZZ^2 \|v\|_3 + \|\eta\|_\ZZ \|\eta\|_3 \|v\|_\ZZ, \\
\|\mathrm{d}^2\KK_\mathrm{nl}^\prime[\eta](v,w)\|_1 & \lesssim \|\eta\|_3 \|v\|_\ZZ\|w\|_\ZZ + \|\eta\|_\ZZ \|v\|_\ZZ \|w\|_3 + \|\eta\|_\ZZ\|v\|_3 \|w\|_\ZZ
\end{align*}
hold for each $\eta \in U$ and $v,w \in H^3(\R^2)$.
\end{lemma}
\begin{proof}
It follows from the formula
\[
\KK_\mathrm{nl}(\eta)=-\beta \int_{\R^2} \frac{(\eta_x^2+\eta_z^2)^2}{2(1+\sqrt{1+\eta_x^2+\eta_z^2})^2}\dx \dz
\]
that
$$\KK_\mathrm{nl}^\prime(\eta)=f_1(\eta_x, \eta_z)\eta_{xx}+f_2(\eta_x, \eta_z)\eta_{xz}+f_3(\eta_x,\eta_z)\eta_{zz},$$
where $f_1, f_2, f_3$ are smooth functions with zeros of order two at the origin, and formulae for the derivatives of $\KK_\mathrm{nl}^\prime$
are in turn derived from this expression. The stated estimates are obtained from these explicit formulae in the usual fashion.\end{proof}

\begin{corollary} \label{cor:tildeNN estimates}
The quantity
$$\BB(\eta_1,\eta_3)=\KK_\mathrm{nl}^\prime(\eta_1+F(\eta_1)+\eta_3)-\Lambda(1-\varepsilon^2)
\big(\LL_4^\prime(\eta_1+F(\eta_1)+\eta_3)+\LL_\mathrm{r}^\prime(\eta_1+F(\eta_1)+\eta_3)\big)$$
satisfies the estimates
\begin{list}{(\roman{count})}{\usecounter{count}}
\item
$\|\BB(\eta_1,\eta_3)\|_1\lesssim (\varepsilon^\theta \nn \eta_1\nn+\|\eta_3\|_3)^2(\nn \eta_1\nn+\|\eta_3\|_3)$,
\item
$\|\mathrm{d}_1\BB[\eta_1,\eta_3]\|_{\LL(\XX_1,H^1(\R^2))} \lesssim (\varepsilon^\theta \nn \eta_1\nn+\|\eta_3\|_3)^2$,
\item
$\|\mathrm{d}_2\BB[\eta_1,\eta_3]\|_{\LL(\XX_2,H^1(\R^2))}\lesssim (\varepsilon^{\theta}\nn \eta_1\nn+\|\eta_3\|_3)(\nn \eta_1\nn+\|\eta_3\|_3)$,
\item
$\|\mathrm{d}_1^2\BB[\eta_1,\eta_3]\|_{\LL^2(\XX_1,H^1(\R^2))}\lesssim \varepsilon^\theta(\varepsilon^\theta\nn \eta_1\nn+\|\eta_3\|_3)$,
\item
$\|\mathrm{d}_1\mathrm{d}_2\BB[\eta_1,\eta_3]\|_{\LL^2(\XX_1 \times \XX_2,H^1(\R^2))}\lesssim \varepsilon^{\theta}\nn \eta_1\nn+ \|\eta_3\|_3$,
\item
$\|\mathrm{d}_2^2\BB[\eta_1,\eta_3]\|_{\LL^2(\XX_2,H^1(\R^2))} \lesssim \nn \eta_1 \nn + \|\eta_3\|_3$
\end{list}
for each $\eta_1\in X_1$ and $\eta_3\in X_3$.
\end{corollary}
\begin{proof}
Writing
$$\tilde{\NN}(\eta) = \KK_\mathrm{nl}^\prime(\eta)-\Lambda(1-\varepsilon^2)\big(\LL_4^\prime(\eta)+\LL_\mathrm{r}^\prime(\eta)\big),$$
one finds that
\begin{align*}
\BB(\eta_1,\eta_3) & = \tilde{\NN}(\eta), \\
\mathrm{d}_1\BB[\eta_1,\eta_3](v_1) & = \mathrm{d}\tilde{\NN}[\eta](v_1 + \mathrm{d}F[\eta_1](v_1)), \\
\mathrm{d}_2\BB[\eta_1,\eta_3](v_3) & = \mathrm{d}\tilde{\NN}[\eta](v_3), \\
\mathrm{d}_1^2\BB[\eta_1,\eta_3](\{v_1\}^2) & = \mathrm{d}^2\tilde{\NN}[\eta](\{v_1 + \mathrm{d}F[\eta_1](v_1)\}^2)
+\mathrm{d}\tilde{\NN}[\eta](\mathrm{d}^2F[\eta_1](\{v_1\}^2)), \\
\mathrm{d}_1\mathrm{d}_2\BB[\eta_1,\eta_3](v_1,v_3) & = \mathrm{d}^2\tilde{\NN}[\eta](v_1 + \mathrm{d}F[\eta_1](v_1),v_3), \\
\mathrm{d}_2^2\BB[\eta_1,\eta_3](v_3) & = \mathrm{d}^2\tilde{\NN}[\eta](\{v_3\}^2),
\end{align*}
and the right-hand sides of these expressions are estimated using the linearity of the derivative,
Lemmata \ref{lem:4 estimates}--\ref{lem:KKnlprime estimates} and Corollary \ref{cor:F estimates}.
\end{proof}

Altogether we have established the following estimates for $G$ and its derivatives
(see Remark \ref{rem:K0Feta1 remark}, Lemma \ref{lem:3 estimates} and Corollary \ref{cor:tildeNN estimates}).

\begin{lemma} \label{lem:Complete G estimates}
The function $G: X_1 \times X_3 \rightarrow \XX_2$ satisfies the estimates
\begin{list}{(\roman{count})}{\usecounter{count}}
\item
$\|G(\eta_1,\eta_3)\|_3\lesssim (\varepsilon^\theta \nn \eta_1\nn+\|\eta_3\|_3)^2(1+\nn \eta_1\nn+\|\eta_3\|_3)+\varepsilon^2\|\eta_3\|_3$,
\item
$\|\mathrm{d}_1G[\eta_1,\eta_3]\|_{\LL(\XX_1,\XX_2)} \lesssim (\varepsilon^\theta \nn \eta_1\nn+\|\eta_3\|_3)
(\varepsilon^\theta+\varepsilon^\theta \nn \eta_1\nn+\|\eta_3\|_3)$,
\item
$\|\mathrm{d}_2G[\eta_1,\eta_3]\|_{\LL(\XX_2,\XX_2)}\lesssim (\varepsilon^{\theta}\nn \eta_1\nn+\|\eta_3\|_3)(1+\nn \eta_1\nn+\|\eta_3\|_3)+\varepsilon^2$,
\item
$\|\mathrm{d}_1^2G[\eta_1,\eta_3]\|_{\LL^2(\XX_1,\XX_2)}\lesssim \varepsilon^\theta(\varepsilon^\theta+\varepsilon^\theta\nn \eta_1\nn+\|\eta_3\|_3)$,
\item
$\|\mathrm{d}_1\mathrm{d}_2G[\eta_1,\eta_3]\|_{\LL^2(\XX_1 \times \XX_2,\XX_2)}\lesssim \varepsilon^\theta+\varepsilon^{\theta}\nn \eta_1\nn+ \|\eta_3\|_3$,
\item
$\|\mathrm{d}_2^2G[\eta_1,\eta_3]\|_{\LL^2(\XX_2,\XX_2)} \lesssim 1+\nn \eta_1 \nn + \|\eta_3\|_3$
\end{list}
for each $\eta_1\in X_1$ and $\eta_3\in X_3$, where $\LL(\XX_1,\XX_2)$, $\LL(\XX_2,\XX_2)$ and
$\LL^2(\XX_1,\XX_2)$, $\LL^2(\XX_1 \times \XX_2,\XX_2)$, $\LL^2(\XX_2,\XX_2)$ denote the Banach spaces
of bounded linear and bilinear operators from the indicated spaces to $\XX_2$.
\end{lemma}

\begin{theorem} \label{thm:estimate eta3}
Equation \eqref{eta3 eqn} has a unique solution $\eta_3 \in
X_3$ which depends smoothly upon $\eta_1 \in X_1$ and satisfies the estimates
$$
\|\eta_3(\eta_1)\|_3 \lesssim \varepsilon^{2\theta} \nn \eta_1\nn^2,  \quad
\|\mathrm{d}\eta_3[\eta_1]\|_{\LL(\XX_1,\XX_2)} \lesssim \varepsilon^{2\theta} \nn \eta_1\nn, \quad
\|\mathrm{d}^2 \eta_3[\eta_1]\|_{\LL^2(\XX_1,\XX_2)} \lesssim \varepsilon^{2\theta}.
$$
\end{theorem}
\begin{proof}
Choosing $R_3$ and $\varepsilon$ sufficiently small, one finds $r>0$ such that
$$\|G(\eta_1,0)\|_3 \leq \frac{r}{2}, \qquad \|\mathrm{d}_2 G[\eta_1,\eta_3]\|_{\LL(\XX_2,\XX_2)} \leq \frac{1}{3}$$
for $\eta_1 \in X_1$, $\eta_3 \in X_3$ (see Lemma \ref{lem:Complete G estimates}(i), (iii)),
and Theorem \ref{thm:fixed-point} asserts that equation \eqref{eta3 eqn} has a unique solution $\eta_3$
in $X_3$ which depends smoothly
upon $\eta_1 \in X_1$. More precise estimates are obtained by choosing $C>0$ so that
$$\|G(\eta_1,0)\|_3 \leq C\varepsilon^{2\theta} \nn \eta_1\nn^2, \qquad \eta_1 \in X_1$$
and writing $r(\eta)=2C\varepsilon^{2\theta}\nn \eta_1\nn^2$, so that
$$\|\mathrm{d}_2 G[\eta_1,\eta_3]\|_{\LL(\XX_2,\XX_2)}\lesssim \varepsilon^\theta, \qquad \eta_1 \in X_1,\ \eta_3 \in \overline{B}_{r(\eta_1)}(0) \subset X_3$$
(Lemma \ref{lem:Complete G estimates}(i), (iii)), and the stated estimates for $\eta_3(\eta_1)$ follow
from Theorem \ref{thm:fixed-point} and Lemma \ref{lem:Complete G estimates}(ii), (iv)--(vi).
\end{proof}

The reduced functional $\widetilde{\JJ}_\varepsilon: X_1 \rightarrow \R$ is defined by
\begin{equation}
\label{tilde J}
\widetilde{\JJ}_\varepsilon(\eta_1):=\JJ_\varepsilon(\eta_1+\eta_2(\eta_1)),
\end{equation}
where $\eta_2(\eta_1)=F(\eta_1)+\eta_3(\eta_1)$ and $\mathrm{d}\JJ_\varepsilon[\eta_1+\eta_2(\eta_1)](\rho_2)=0$ for all $\rho_2 \in \XX_2$ by construction.
It follows that
\begin{align*}
\mathrm{d}\widetilde{\JJ}_\varepsilon[\eta_1](\rho_1)&=\mathrm{d}\JJ_\varepsilon[\eta_1+\eta_2(\eta_1)](\rho_1)+\mathrm{d}\JJ_\varepsilon[\eta_1+\eta_2(\eta_1)](\mathrm{d}\eta_2[\eta_1](\rho_1))\\
&=\mathrm{d}\JJ_\varepsilon[\eta_1+ \eta_2(\eta_1)](\rho_1)
\end{align*}
for all $\rho_1 \in \XX_1$, so that each critical point $\eta_1$ of $\widetilde{\JJ}_\varepsilon$ defines a critical point $\eta_1+\eta_2(\eta_1)$ of $\JJ_\varepsilon$.
Conversely, each critical point $\eta=\eta_1+\eta_2$ of $\JJ_\varepsilon$ with $\eta_2-F(\eta_1) \in X_3$ has the properties
that $\eta_2=\eta_2(\eta_1)$ and $\eta_1$ is a critical point of $\widetilde{\JJ}_\varepsilon$.

\bigskip

\subsection{The reduced functional}
\label{section: expansion}

In this section we compute leading-order terms in the reduced functional
\begin{align}
\widetilde{\JJ}_\varepsilon(\eta_1) &= \HH(\eta) + \Lambda\varepsilon^2\LL_2(\eta)+\NN(\eta) \nonumber \\
& = \HH(\eta) + \Lambda\varepsilon^2\LL_2(\eta) 
+\KK_4(\eta)+\KK_\mathrm{r}(\eta)
- \Lambda (1-\varepsilon^2) \big(\LL_3(\eta)+\LL_4(\eta)+\LL_\mathrm{r}(\eta)\big),
\label{Red J expansion}
\end{align}
where
$$
\HH(\eta):=\KK_2(\eta)-\Lambda\LL_2(\eta)=\frac{1}{2}\int_{\R^2} \tilde{g}(k) |\hat{\eta}|^2 \dk_1\dk_2
$$
and $\eta=\eta_1+F(\eta_1)+\eta_3(\eta_1)$.
Writing
$$\eta_1 = \eta_1^+ + \eta_1^-,$$
where
$$\eta_1^+ = \FF^{-1}[\chi^+\hat{\eta}_1], \qquad \eta_1^- = \FF^{-1}[\chi^-\hat{\eta}_1] = \overline{\eta_1^+}$$
and $\chi^+$, $\chi^-$ are the characteristic functions of respectively $B_\delta(\omega,0)$ and
$B_\delta(-\omega,0)$, we establish the following theorem.

\begin{theorem} \label{thm:red func}
The reduced functional is given by the formula
\begin{align*}
\widetilde{\JJ}_\varepsilon(\eta_1)&=\int_{\R^2} \tilde{g}(k) |\FF[\eta_1^+]|^2 \dk_1\dk_2 + \varepsilon^2\Lambda f(\omega) \int_{\R^2} |\eta_1^+|^2 \dx \dz \\
&\qquad\mbox{} -16C_1\int_{\R^2}\frac{k_1^2}{(1-\Lambda)k_1^2+k_2^2} |\FF[|\eta_1^+|^2]|^2\dk_1\dk_2-16C_2\int_{\R^2}|\eta_1^+|^4 \dx\dz\\
&\qquad\mbox{}+\underline{\OO}(\varepsilon^{3\theta}\nn \eta_1\nn^2),
\end{align*}
where
\begin{align*}
C_1&=\frac{\Lambda}{32} (\Lambda B(\omega)-2f(\omega))^2, \\
C_2&=\frac{g(2\omega)^{-1}\Lambda^2 A(\omega)^2}{16}+\frac{\Lambda^2 B(\omega)^2}{32}+\frac{3\beta \omega^4}{64}
+\frac{\Lambda f(\omega)}{16} (f(\omega) f(2\omega)-3 \omega^2),
\end{align*}
\[
A(\omega)=\frac{3\omega^2-f(\omega)^2-2f(\omega) f(2\omega)}{2}, \qquad
B(\omega)=\omega^2-f(\omega)^2
\]
and the symbol $\underline{\OO}(\varepsilon^\gamma \nn \eta_1\nn^r)$ (with $\gamma \geq 0$, $r \geq 1$)
denotes a smooth functional $\RR: X_1 \rightarrow \R$ which satisfies the estimates 
$$
|\RR(\eta_1)| \lesssim \varepsilon^\gamma  \nn \eta_1\nn^r, \quad
\|\mathrm{d}\RR[\eta_1]\|_{\LL(\XX_1, \R)}\lesssim \varepsilon^\gamma  \nn \eta_1 \nn^{r-1}, \quad
\|\mathrm{d}^2\RR[\eta_1]\|_{\LL^2(\XX_1, \R)}\lesssim \varepsilon^\gamma  \nn \eta_1 \nn^{\max(r-2,0)}
$$
for each $\eta \in X_1$.
\end{theorem}

\begin{remark}
The coefficient $C_1$ is obviously positive, while the positivity of $C_2$ is established by
elementary arguments after substituting
\[
\beta=\frac{f^\prime(\omega)}{2\omega f(\omega)-\omega^2 f^\prime(\omega)}, \qquad
\Lambda=\frac{2\omega}{2\omega f(\omega)-\omega^2 f^\prime(\omega)}
\]
(see the comments to equation \eqref{Defn of g}).
\end{remark}

We begin the proof of Theorem \ref{thm:red func} with a result which shows how 
Fourier-multiplier operators acting upon the function $\eta_1$ may be approximated by constants.

\begin{lemma}  \label{lem:approximate identities}
The estimates
\begin{list}{(\roman{count})}{\usecounter{count}}
\item
$\partial_x \eta_1^\pm = \pm \mathrm{i}\omega \eta_1^\pm + \underline{O}(\varepsilon \nn \eta_1\nn)$,
\item
$\partial_x^2 \eta_1^\pm =-\omega^2\eta_1^\pm+  \underline{O}(\varepsilon \nn \eta_1\nn)$,
\item
$\partial_z \eta_1^\pm =  \underline{O}(\varepsilon \nn \eta_1\nn)$,
\item
$K_0 \eta_1^\pm =f(\omega)\eta_1^\pm + \underline{O}(\varepsilon\nn \eta_1\nn)$,
\item
$L_0 \eta_1^\pm = \underline{O}(\varepsilon \nn \eta_1\nn)$,
\item
$K_0((\eta_1^\pm)^2) = f(2\omega)(\eta_1^\pm)^2 +  \underline{O}(\varepsilon^{1+\theta}\nn \eta_1\nn^2)$,
 \item
$L_0((\eta_1^\pm)^2) =  \underline{O}(\varepsilon^{1+\theta}\nn \eta_1\nn^2)$,
\item
$K_0 (\eta_1^+\eta_1^-) =\FF^{-1}[k_1^2/|k|^2\FF[\eta_1^+\eta_1^-]] + \underline{O}(\varepsilon^{1+\theta}\nn \eta_1\nn^2)$,
\item
$\FF^{-1}[\tilde g(k)^{-1}\FF[ (\eta_1^\pm)^2]]=g(2\omega)^{-1} (\eta_1^\pm)^2+ \underline{O}(\varepsilon^{1+\theta}\nn \eta_1\nn^2)$,
\item
$\FF^{-1}[(\tilde g(k)^{-1}-(1-\Lambda k_1^2/|k|^2)^{-1})\FF[ \eta_1^+\eta_1^-]]=\underline{O}(\varepsilon^{1+\theta}\nn \eta_1\nn^2)$
\end{list}
hold for each $\eta_1 \in X_1$,
where the symbol $\underline{O}(\varepsilon^\gamma \nn \eta_1\nn^r)$ (with $\gamma \geq 0$, $r \geq 1$)
denotes a smooth function $R: X_1 \rightarrow H^1(\R^2)$ whose Fourier transform has support which lies in a compact set
whose size does not depend upon $\varepsilon$ and which satisfies the estimates 
$$
\|R(\eta_1)\|_1 \lesssim \varepsilon^\gamma  \nn \eta_1 \nn^r,$$
$$
\|\mathrm{d}R[\eta_1]\|_{\LL(\XX_1, H^1(\R^2))}\lesssim \varepsilon^\gamma  \nn \eta_1 \nn^{r-1}, \qquad
\|\mathrm{d}^2R[\eta_1]\|_{\LL^2(\XX_1, H^1(\R^2))}\lesssim \varepsilon^\gamma  \nn \eta_1\nn^{\max(r-2,0)}
$$
for each $\eta_1 \in X_1$. (One may replace $H^1(\R^2)$ with $H^s(\R^2)$ for any $s \geq 0$
in these estimates.)
\end{lemma}
\begin{proof}
Note that
\begin{align*}
\|\partial_x \eta_1^\pm \mp \mathrm{i}\omega \eta_1^\pm\|_0^2
&\le \int_{\R^2} (|k_1|-\omega)^2 |\hat \eta_1(k)|^2 \dk_1 \dk_2\\
&\le \varepsilon^2\int_{\R^2} (1+\varepsilon^{-2}((|k_1|-\omega)^2+k_2^2)) |\hat \eta_1(k)|^2 \dk_1 \dk_2\\
&=\varepsilon^{2}\nn \eta_1\nn^2,
\end{align*}
and iterating this argument yields (ii); similarly
\begin{align*}
\|\partial_z \eta_1^\pm\|_0^2
&\le \int_{\R^2} k_2^2 |\hat \eta_1(k)|^2 \dk_1 \dk_2\\
&\le \varepsilon^2\int_{\R^2} (1+\varepsilon^{-2}((|k_1|-\omega)^2+k_2^2)) |\hat \eta_1(k)|^2 \dk_1 \dk_2\\
&=\varepsilon^{2}\nn \eta_1\nn^2.
\end{align*}
Moreover, the functions $\hat K_0(k)=(k_1^2/|k|^2) f(|k|)$ and $\hat L_0(k)=(k_1k_2/|k|^2) f(|k|)$ are smooth at
the points $(\pm \omega,0)$ with $\hat K_0(\pm \omega,0)=f(\omega)$ and $\hat L_0(\pm \omega,0)=0$, so that
\begin{align*}
\|K_0\eta_1^\pm-f(\omega)\eta_1^\pm\|_0^2
&\lesssim \int_{\R^2} ((|k_1|-\omega)^2 +k_2^2)|\hat \eta_1(k)|^2 \dk_1 \dk_2\\
&\le \varepsilon^{2}\nn \eta_1\nn^2, \\
\\
\|L_0\eta_1^\pm\|_0^2
&\lesssim \int_{\R^2} ((|k_1|-\omega)^2 +k_2^2)|\hat \eta_1(k)|^2 \dk_1 \dk_2\\
&\le \varepsilon^{2}\nn \eta_1\nn^2.
\end{align*}

Notice that the quantities to be estimated in (vi)--(x) are quadratic in $\eta_1$;
it therefore suffices to estimate the corresponding bilinear operators. To this end we take $v_1\in \XX_1$ and define $v_1^\pm$ in the same
way as $\eta_1^\pm$. The argument used for (iv) and (v) above yields
\begin{align*}
|\FF[K_0(\eta_1^+ v_1^+)-f(2\omega)\eta_1^+ v_1^+]| &\lesssim |k-(2\omega,0)|
\int_{\R^2} | \hat \eta_1^+(k-s)| | \hat v_1^+(s)| \ds_1 \ds_2\\
&\lesssim \int_{\R^2}  |k-s-(\omega,0)|  | \hat \eta_1^+(k-s)|  | \hat v_1^+(s)| \ds_1 \ds_2\\
&\qquad\mbox{} +\int_{\R^2} |s-(\omega,0)| | \hat \eta_1^+(k-s)| | \hat v_1^+(s)| \ds_1 \ds_2,
\end{align*}
and using Young's inequality, we find that
\begin{align*}
\|K_0(\eta_1^+ v_1^+)-f(2\omega)\eta_1^+ v_1^+\|_0& \lesssim \|\ |k-(\omega,0)|\hat \eta_1^+\ \|_0 \|\hat v_1^+\|_{L^1(\R^2)}+\|\hat \eta_1^+\|_{L^1(\R^2)}
\|\ |k -(\omega,0)|\hat v_1^+\ \|_0 \\
&\lesssim \varepsilon^{1+\theta}\nn \eta_1 \nn\,  \nn v_1\nn.
\end{align*}
The corresponding results for $K_0 (\eta_1^-v_1^-)$ and $L_0 (\eta_1^\pm v_1^\pm)$ are obtained in a similar fashion.

Turning to (viii), we note that 
\begin{align*}
\left|\FF[K_0 (\eta_1^+v_1^-)]-\frac{k_1^2}{|k|^2}\FF[\eta_1^+v_1^-]\right|
& \lesssim \frac{k_1^2}{|k|^2}|f(|k|)-f(0)| \left|\int_{\R^2} \hat \eta_1^+(k-s) \hat v_1^-(s)\ds_1 \ds_2 \right|\\
&\lesssim |k|
\int_{\R^2} | \hat \eta_1^+(k-s)| | \hat v_1^-(s)| \ds_1 \ds_2\\
&\lesssim \int_{\R^2}  |k-s-(\omega,0)|  | \hat \eta_1^+(k-s)|  | \hat v_1^-(s)| \ds_1 \ds_2\\
&\qquad \mbox{}+\int_{\R^2} |s+(\omega,0)| | \hat \eta_1^+(k-s)| | \hat v_1^-(s)| \ds_1 \ds_2,
\end{align*}
whence
$$\left\|K_0(\eta_1^+ v_1^-)-\FF^{-1}\left[\frac{k_1^2}{|k|^2}\FF[\eta_1^+v_1^-]\right]\right\|_0
\lesssim \varepsilon^{1+\theta}\nn \eta_1 \nn \,\nn v_1\nn$$
(by Young's inequality).
Estimate (ix) follows from the calculation
\begin{align*}
\|\FF^{-1}[\tilde g(k)^{-1}\FF[ \eta_1^+ v_1^+]]-\tilde g(2\omega,0)^{-1} \eta_1^+ v_1^+\|_0&= \|\FF^{-1}[(\tilde g(k)^{-1}-\tilde g(2\omega,0)^{-1})\FF[ \eta_1^+ v_1^+ ]]\|_0\\
&\lesssim \| | k- (2\omega,0)| \FF[ \eta_1^+ v_1^+]\|_0\\
&\lesssim \varepsilon^{1+\theta}\nn \eta_1 \nn \, \nn v_1\nn
\end{align*}
and the identity $\tilde g(2\omega,0)=g(2\omega)$
(with a similar argument for $\eta_1^-v_1^-$), while (x) is a consequence of the calculation
\begin{align*}
\left\|\tilde g(k)^{-1}\FF[ \eta_1^+ v_1^-]-\left(1-\Lambda \frac{k_1^2}{|k|^2}\right)^{-1} \FF[\eta_1^+ v_1^-]\right\|_0&\lesssim \| |k| \FF[ \eta_1^+ v_1^-]\|_0\\
&\lesssim \varepsilon^{1+\theta}\nn \eta_1 \nn \,\nn v_1\nn. \qedhere
\end{align*}
\end{proof}

The next step is to derive an approximate formula for $\HH(F(\eta_1))$.

\begin{proposition}
The estimate
$$
\LL_3^\prime(\eta_1)=A(\omega)((\eta_1^+)^2+(\eta_1^-)^2)+ B(\omega) \eta_1^+\eta^-_1-2 f(\omega)  \FF^{-1}\left[\frac{k_1^2}{|k|^2}\FF[\eta_1^+\eta_1^-]\right]  +\underline{O}(\varepsilon^{1+\theta}\nn \eta_1\nn^2)
$$
holds for each $\eta_1 \in X_1$.
\end{proposition}

\begin{proof}
One obtains the stated estimate by combining the formula
$$
\LL_3^\prime(\eta)=-\frac{1}{2}\left(\eta_{x}^2+2\eta\eta_{xx}+  (K_0 \eta)^2 +(L_0 \eta)^2+2K_0(\eta K_0\eta)+2L_0(\eta L_0\eta)
\right)
$$
(see Lemma \ref{lem:L2 to L4}) with the calculations
\begin{align*}
\eta_{1x}^2 &= (\mathrm{i}\omega \eta_1^+-\mathrm{i}\omega \eta_1^-+\underline{O}(\varepsilon \nn \eta_1\nn))^2
=-\omega^2 (\eta_1^+- \eta_1^-)^2 +\underline{O}(\varepsilon^{1+\theta} \nn \eta_1\nn^2), \\[1mm]
(K_0 \eta_1)^2 &= (f(\omega)\eta_1+\underline{O}(\varepsilon \nn \eta_1\nn))^2
=f(\omega)^2 \eta_1^2+\underline{O}(\varepsilon^{1+\theta} \nn \eta_1\nn^2), \\[1mm]
(L_0 \eta_1)^2 &= \underline{O}(\varepsilon \nn \eta_1\nn)^2
=\underline{O}(\varepsilon^{2} \nn \eta_1\nn^2), \\[1mm]
\eta_1 \eta_{1xx} &=\eta_1(-\omega^2 \eta_1+ \underline{O}(\varepsilon \nn \eta_1\nn)
=-\omega^2\eta_1^2 +\underline{O}(\varepsilon^{1+\theta} \nn \eta_1\nn^2), \\[1mm]
\eta_1 K_0 \eta_1 &= \eta_1 (f(\omega)\eta_1+\underline{O}(\varepsilon \nn \eta_1\nn))
=f(\omega) \eta_1^2+\underline{O}(\varepsilon^{1+\theta} \nn \eta_1\nn^2), \\[1mm]
\eta_1 L_0 \eta_1 &= \eta_1 \underline{O}(\varepsilon \nn \eta_1\nn)
=\underline{O}(\varepsilon^{1+\theta} \nn \eta_1\nn^2)
\end{align*}
and
\begin{align*}
K_0(\eta_1 K_0 \eta_1) &= K_0(f(\omega) \eta_1^2+\underline{O}(\varepsilon^{1+\theta} \nn \eta_1\nn^2))\\
&=f(\omega) K_0 (\eta_1^2)+\underline{O}(\varepsilon^{1+\theta} \nn \eta_1\nn^2)\\
&=f(\omega) \left( K_0((\eta_1^+)^2)+2K_0(\eta_1^+\eta_1^-)+K_0((\eta_1^-)^2)\right)+\underline{O}(\varepsilon^{1+\theta} \nn \eta_1\nn^2)\\
&=f(\omega)f(2\omega)\left((\eta_1^+)^2+(\eta_1^-)^2\right)+2f(\omega) \FF^{-1}\left[\frac{k_1^2}{|k|^2}\FF[\eta_1^+\eta_1^-]\right]
+\underline{O}(\varepsilon^{1+\theta} \nn \eta_1\nn^2), \\[1mm]
L_0(\eta_1 L_0 \eta_1) &= L_0(\underline{O}(\varepsilon^{1+\theta} \nn \eta_1\nn^2))\\
&=\underline{O}(\varepsilon^{1+\theta} \nn \eta_1\nn^2)
\end{align*}
(see Lemma \ref{lem:approximate identities}).
\end{proof}

\begin{corollary}
The estimate
\begin{eqnarray*}
\lefteqn{\HH(F(\eta_1))}\\
&&=  \int_{\R^2}\!\!\left(\Lambda^2g(2\omega)^{-1} A(\omega)^2+\frac{\Lambda^2B(\omega)^2}{2}\right) |\eta_1^+|^4 \dx\dz
-2 \Lambda f(\omega)^2 \int_{\R^2} \frac{k_1^2}{|k|^2}|\FF[|\eta_1^+|^2]|^2\dk_1\dk_2
\\
&&\qquad\mbox{} +\frac{\Lambda (\Lambda B(\omega)-2f(\omega))^2}2 \int_{\R^2}\frac{k_1^2}{(1-\Lambda)k_1^2+k_2^2}|\FF[|\eta_1^+|^2]|^2\dk_1\dk_2+\underline{O}(\varepsilon^{1+2\theta}\nn \eta_1\nn^4)
\end{eqnarray*}
holds for each $\eta_1 \in X_1$.
\end{corollary}

\begin{proof}
The result follows from the calculation
\begin{eqnarray*}
\lefteqn{\HH(F(\eta_1))} \\
&&= \frac12 \int_{\R^2} \tilde g(k) |\FF[F(\eta_1)]|^2\dk_1 \dk_2\\
&&= \frac12 \Lambda^2(1-\varepsilon)^2\int_{\R^2} \tilde g(k)^{-1} |\FF[\LL_3^\prime(\eta_1)]|^2 \dk_1 \dk_2\\
&&=\frac12 \Lambda^2(1-\varepsilon)^2\int_{\R^2} \tilde g(k)^{-1} \Big|\FF\Big[A(\omega)((\eta_1^+)^2+(\eta_1^-)^2)+ B(\omega) \eta^+\eta^-\\
& &\hspace{2in} -2 f(\omega)  \FF^{-1}\left[\frac{k_1^2}{|k|^2}\FF[\eta_1^+\eta_1^-]\right] +\underline{O}(\varepsilon^{1+\theta}\nn \eta_1\nn^2)\Big]\Big|^2 \dk_1 \dk_2\\
&&=\frac12 \Lambda^2\int_{\R^2} \tilde g(k)^{-1} A(\omega)^2(|\FF[(\eta_1^+)^2]|^2+|\FF[(\eta_1^-)^2]|^2)\dk_1\dk_2\\
& &\qquad\mbox{}+\frac12 \Lambda^2\int_{\R^2} \tilde g(k)^{-1} \left(B(\omega)-2 f(\omega) \frac{k_1^2}{|k|^2} \right)^2 |\FF[\eta_1^+\eta_1^-]|^2\dk_1\dk_2
+\underline{O}(\varepsilon^{1+2\theta}\nn \eta_1\nn^4)\\
&&=\frac12 \Lambda^2\int_{\R^2} g(2\omega)^{-1} A(\omega)^2(|\FF[(\eta_1^+)^2]|^2+|\FF[(\eta_1^-)^2]|^2)\dk_1\dk_2\\
& &\qquad\mbox{}+\frac12\Lambda^2 \int_{\R^2} \left(1-\Lambda \frac{k_1^2}{|k|^2}\right)^{-1} \left(B(\omega)-2 f(\omega) \frac{k_1^2}{|k|^2} \right)^2 |\FF[\eta_1^+\eta_1^-]|^2\dk_1\dk_2\\
& &\qquad\mbox{} + \underline{O}(\varepsilon^{1+2\theta}\nn \eta_1\nn^4)\\
&&=  \int_{\R^2} \Lambda^2g(2\omega)^{-1} A(\omega)^2 |\eta_1^+|^4 \dx\dz\\
& &\qquad\mbox{} +\frac12\Lambda^2 \int_{\R^2} \left(1-\Lambda \frac{k_1^2}{|k|^2}\right)^{-1} \left(B(\omega)-2 f(\omega) \frac{k_1^2}{|k|^2} \right)^2 |\FF[|\eta_1^+|^2]|^2\dk_1\dk_2\\
& &\qquad\mbox{} + \underline{O}(\varepsilon^{1+2\theta}\nn \eta_1\nn^4)
\end{eqnarray*}
and the identity
$$
\Lambda^2\left(1-\Lambda \frac{k_1^2}{|k|^2}\right)^{-1} \left(\!B(\omega)-2 f(\omega) \frac{k_1^2}{|k|^2} \right)^2
\!\!= \Lambda^2B(\omega)^2-\frac{4 \Lambda f(\omega)^2k_1^2}{|k|^2}+\frac{\Lambda (\Lambda B(\omega)-2f(\omega))^2k_1^2}{(1-\Lambda)k_1^2+k_2^2}.
$$
\end{proof}

We now examine systematically each term on the right-hand side of equation \eqref{Red J expansion}.

\begin{lemma} \label{lem:Expansion 1}
The estimate
$$
\HH(\eta_1+ F(\eta_1) + \eta_3(\eta_1)) = \frac{1}{2}\int_{\R^2} \tilde{g}(k) |\hat{\eta}_1|^2 \dk_1 \dk_2 + \HH(F(\eta_1))  + \underline{\OO}(\varepsilon^{3\theta}\nn \eta_1 \nn^4)
$$
holds for each $\eta_1 \in X_1$.
\end{lemma}
\begin{proof}
Observe that
$$
\HH(\eta)= \HH(\eta_1)+\HH(F(\eta_1)) + \underbrace{\HH((\eta_3(\eta_1))}_{\displaystyle =\underline{\OO}(\varepsilon^{4\theta}\nn \eta_1 \nn^4)}
 + \underbrace{\int_{\R^2} \tilde{g}(k) \FF[\eta_3(\eta_1)] \overline{\FF[F(\eta_1)]} \dk_1 \dk_2}_{\displaystyle =\underline{\OO}(\varepsilon^{3\theta}\nn \eta_1 \nn^4)},
$$
where we have used the facts that $\supp \hat{\eta}_1 \cap \supp \hat{\eta}_3(\eta_1) = \emptyset$,
$\supp \hat{\eta}_1 \cap \supp \FF[F(\eta_1)] = \emptyset$ to obtain the equation,
and Corollary \ref{cor:F estimates}, Theorem \ref{thm:estimate eta3} and the inequality $\tilde{g}(k) \lesssim 1+|k|^2$
to obtain the estimates.
\end{proof}

\begin{lemma}
The estimate
$$
\LL_2(\eta_1+ F(\eta_1) + \eta_3(\eta_1))=f(\omega)\int_{\R^2}|\eta_1^+|^2\dx\dz + \underline{\OO}(\varepsilon \nn \eta_1 \nn^2)
+ \underline{\OO}(\varepsilon^{2\theta} \nn \eta_1 \nn^4)
$$
holds for each $\eta_1 \in X_1$.
\end{lemma}

\begin{proof}
Repeating the argument used in the proof of the previous lemma, we find that
$$\LL_2(\eta)= \LL_2(\eta_1)+\LL_2(F(\eta_1)) + \underline{\OO}(\varepsilon^{3\theta}\nn \eta_1 \nn^4),$$
while
$$\LL_2(\eta_1) = \frac{1}{2}f(\omega)\int_{\R^2}|\eta_1|^2\dx\dz + \underline{\OO}(\varepsilon \nn \eta_1 \nn^2)$$
and
\[
\LL_2(F(\eta_1)) = \underline{\OO}(\varepsilon^{2\theta} \nn \eta_1 \nn^4).\qedhere
\]
\end{proof}

\begin{lemma}
The estimate
$$
\LL_3(\eta_1+ F(\eta_1) + \eta_3(\eta_1))=\frac{2}{\Lambda(1-\varepsilon^2)}\HH(F(\eta_1))+\underline{\OO}(\varepsilon^{3\theta} \nn \eta_1\nn^4)
$$
holds for each $\eta_1 \in X_1$.
\end{lemma}

\begin{proof}
Observe that
$$\LL_3(\eta) = p_\mathrm{sym}(\{\eta\}^3),$$
where
$$p_\mathrm{sym}(u_1,u_2,u_3)=\frac{1}{6}\sum_{\sigma \in S_3} p(u_{\sigma(1)},u_{\sigma(2)}, u_{\sigma(3)})
$$
and
$$
p(u,v,w)=\frac{1}{2}\int_{\R^2} \left(u_x v_x w - u K_0v K_0w - u L_0v L_0w \right)\dx\dz
$$
(see Lemma \ref{lem:L2 to L4}).

Using the estimate
$$|p(u,v,w)| \lesssim \varepsilon^\theta \nn u \nn \|v\|_3 \|w\|_3$$
for $u \in \XX_1$ and $v,w \in H^3(\R^2)$, Corollary \ref{cor:F estimates} and Theorem \ref{thm:estimate eta3}, we find that
$$\LL_3(\eta_1+ F(\eta_1) + \eta_3(\eta_1)) = p(\{\eta_1\}^3) + 3 p(\{\eta_1\}^2, F(\eta_1)) + \underline{\OO}(\varepsilon^{3\theta} \nn \eta_1\nn^4).$$
On the other hand
$$p(\{\eta_1\}^3) = \LL_3(\eta_1) = \frac{1}{3}\langle \LL_3^\prime(\eta_1), \eta_1 \rangle_0 =0$$
and
\begin{align*}
3p(\{\eta_1\}^2, F(\eta_1)) &=\langle \LL_3^\prime(\eta_1), F(\eta_1)\rangle_0 \\
&=\frac{1}{\Lambda(1-\varepsilon^2)}\langle \tilde{g}(D) F(\eta_1), F(\eta_1) \rangle_0\\
&=\frac{2}{\Lambda(1-\varepsilon^2)}\HH(F(\eta_1)).\qedhere
\end{align*}
\end{proof}

\begin{lemma}
The estimates
\begin{align*}
\KK_4(\eta_1+ F(\eta_1)+\eta_3(\eta_1))
&=-\frac{3\beta \omega^4}{4}\int_{\R^2}|\eta_1^+|^4 \dx \dz+
 \underline{\OO}(\varepsilon^{3\theta} \nn \eta_1\nn^4)
\end{align*}
and
\begin{align*}
\LL_4(\eta_1+ F(\eta_1)+\eta_3(\eta_1))
&=f(\omega) (f(\omega) f(2\omega)-3 \omega^2) \int_{\R^2} |\eta_1^+|^4\dx \dz \\
&\qquad \mbox{}+
2f(\omega)^2 \int_{\R^2} \frac{k_1^2}{|k|^2} |\FF[|\eta_1^+|^2]|^2\dk_1 \dk_2
+ \underline{\OO}(\varepsilon^{3\theta} \nn \eta_1\nn^4).
\end{align*}
hold for each $\eta_1 \in X_1$.
\end{lemma}

\begin{proof}
Proceeding as in the proof of the previous lemma, one finds that 
$$
\KK_4(\eta_1+ F(\eta_1)+\eta_3(\eta_1))=
\KK_4(\eta_1)+\underline{\OO}(\varepsilon^{3\theta} \nn \eta_1\nn^5),
$$
and it follows from the rules given in Lemma \ref{lem:approximate identities} that
\begin{align*}
\KK_4(\eta_1)&=-\frac{\beta} 8\int_{\R^2}(\eta_{1x}^2+\eta_{1z}^2)^2\dx\dz\\
&=-\frac{\beta} 8\int_{\R^2}\eta_{1x}^4\dx\dz+
 \underline{\OO}(\varepsilon^{3\theta} \nn \eta_1\nn^4)\\
&=-\frac{\beta \omega^4} 8\int_{\R^2}(\eta_{1}^+-\eta_{1}^-)^4\dx\dz+
 \underline{\OO}(\varepsilon^{3\theta} \nn \eta_1\nn^4)\\&=-\frac{3\beta \omega^4}{4}\int_{\R^2}(\eta_1^+)^2(\eta_1^-)^2 \dx \dz+ \underline{\OO}(\varepsilon^{3\theta} \nn \eta_1\nn^4)\\
&=-\frac{3\beta \omega^4}{4}\int_{\R^2}|\eta_1^+|^4 \dx \dz+ \underline{\OO}(\varepsilon^{3\theta} \nn \eta_1\nn^4).
\end{align*}

The estimate for $\LL_4(\eta_1+ F(\eta_1)+\eta_3(\eta_1))$ is derived in a similar fashion.
\end{proof}

\begin{remark}
Note that $\KK_4(\eta_1)$, $\LL_4(\eta_1)=\underline{\OO}(\varepsilon^{2\theta} \nn \eta_1\nn^4)$
for each $\eta_1 \in X_1$.
\end{remark}

\begin{lemma}
The estimates
$$\KK_\mathrm{r}(\eta_1+F(\eta_1) + \eta_3(\eta_1)) = \underline{\OO}(\varepsilon^{4\theta}\nn \eta_1 \nn^6),
\qquad 
\LL_\mathrm{r}(\eta_1+F(\eta_1) + \eta_3(\eta_1)) = \underline{\OO}(\varepsilon^{3\theta}\nn \eta_1 \nn^5)
$$
hold for each $\eta_1 \in X_1$.
\end{lemma}

\begin{proof}
Observe that
\begin{align*}
|\LL_\mathrm{r}(\eta)| & \lesssim \|\eta\|_{\ZZ}^3 \|\eta\|_3^2, \\
|\mathrm{d}\LL_\mathrm{r}[\eta](v)| & \lesssim \|\eta\|_{\ZZ}^3 \|\eta\|_3 \|v\|_3 + \|\eta\|_{\ZZ}^2 \|\eta\|_3^2\|v\|_\ZZ, \\
|\mathrm{d}^2\LL_\mathrm{r}[\eta](v,w)| & \lesssim \|\eta\|_{\ZZ}^3 \|v\|_3 \|w\|_3 + \|\eta\|_\ZZ \|\eta\|_3^2 \|v\|_\ZZ \|w\|_\ZZ\\
& \qquad\mbox{}+\|\eta\|_{\ZZ}^2 \|\eta\|_3 \|v\|_3 \|w\|_\ZZ + \|\eta\|_{\ZZ}^2\|\eta\|_3 \|v\|_\ZZ \|w\|_3
\end{align*}
for $\eta \in U$ and $v,w \in H^3(\R^2)$  (see Lemmata \ref{lem:u analytic} and \ref{lem:gradients}) and since
$$\KK_\mathrm{r}(\eta) = \int_{\R^2} f(\eta_x,\eta_z)\dx\dz,$$
where $f$ is analytic at the origin where it has a zero of order six, it satisfies similar estimates (with the power of $\|\eta\|_\ZZ$
increased by one). The stated result follows from this observation, Corollary \ref{cor:F estimates} and Theorem \ref{thm:estimate eta3}.
\end{proof}

Theorem \ref{thm:red func} is proved by inserting the above estimates into the right-hand side
of \eqref{Red J expansion}. The next step is to convert $\widetilde{\JJ}_\varepsilon$ into a perturbation of the Davey-Stewartson
functional, the main issue being the replacement of $\tilde{g}(k)$ by its second-order Taylor polynomial at the point $(\omega,0)$, that is
$$
\tilde{g}_2(k) = \frac{1}{2}\partial_{k_1}^2\tilde{g}(\omega,0)(k_1-\omega)^2+\frac{1}{2}\partial_{k_2}^2\tilde{g}(\omega,0)k_2^2.
$$
Using the simple inequality $|\tilde{g}(k)-\tilde{g}_2(k)| \lesssim |k-(\omega,0)|^3$ for $k \in \supp \eta_1^+$ leads to the insufficient estimate
$$\int_{\R^2} |\tilde{g}(k)-\tilde{g}_2(k)||\eta_1^+|^2 \dk_1 \dk_1 = \underline{\OO}(\varepsilon^2 \nn \eta_1 \nn^2),$$
(at the next step the functional is scaled by $\varepsilon^{-2}$). The desired effect is however
achieved using the change of variable
$$
\eta_1 = \left(\frac{\tilde{g}_2(D)}{\tilde{g}(D)}\right)^{1/2}\tilde{\eta}_1
$$
(which defines an isomorphism $\chi(D)H^1(\R) \rightarrow \chi(D)H^1(\R)$).

\begin{lemma} \label{lem:cov to DS}
The reduced functional is given by the formula
\begin{align*}
\widetilde{\JJ}_\varepsilon(\eta_1(\tilde{\eta}_1))&=\int_{\R^2} \tilde{g}_2(k) |\FF[\tilde{\eta}_1^+]|^2 \dk_1\dk_2 + \varepsilon^2\Lambda f(\omega) \int_{\R^2} |\tilde{\eta}_1^+|^2 \dx \dz \\
&\qquad\mbox{}-16C_1\int_{\R^2}\frac{k_1^2}{(1-\Lambda)k_1^2+k_2^2} |\FF[|\tilde{\eta}_1^+|^2]|^2\dk_1\dk_2-16C_2\int_{\R^2}|\tilde{\eta}_1^+|^4 \dx\dz \\
& \qquad\mbox{}+\underline{\OO}(\varepsilon^{3\theta}\nn \tilde{\eta}_1\nn^2).
\end{align*}
\end{lemma}

\begin{proof}
The inequality
\[
\|\eta_1^+|^4-|\tilde{\eta}_1^+|^4|\lesssim |\eta_1^+-\tilde{\eta}_1^+|(|\eta_1^+|^3+|\tilde{\eta}_1^+|^3)
\]
implies that
\begin{align*}
\left|\int_{\R^2} |\eta_1^+|^4 \dx \dz -\int_{\R^2} |\tilde{\eta}_1^+|^4 \dx \dz \right| &\le 
(\|\eta_1^+\|_\infty^2+\|\tilde{\eta}_1^+\|_\infty^2)(\|\eta_1^+\|_0+\|\tilde{\eta}_1^+\|_0)
\|\eta_1^+-\tilde{\eta}_1^+\|_0\\
&\lesssim \varepsilon^{2\theta}\nn \eta_1\nn^3 \left\|\left(1-\left(\frac{\tilde{g}(D)}{\tilde{g}_2(D)}\right)^{1/2}\right)\hat \eta_1^+\right\|_0\\
&\lesssim \varepsilon^{2\theta}\nn \eta_1\nn^3 \|\ |k-(\omega,0)|\hat \eta_1^+\ \|_0\\
&\lesssim \varepsilon^{1+2\theta}\nn \eta_1\nn^4,
\end{align*}
and estimating derivatives in a similar way, we find that
$$\int_{\R^2} |\eta_1^+|^4 \dx \dz =\int_{\R^2} |\tilde{\eta}_1^+|^4 \dx \dz+
\underline{O}(\varepsilon^{1+2\theta}\nn \eta_1 \nn^4).$$
Similar arguments show that
\begin{align*}
\int_{\R^2} \frac{k_1^2}{|k|^2} |\FF[|\eta_1^+|^2]|^2 \dk_1 \dk_2 &=\int_{\R^2} \frac{k_1^2}{|k|^2} |\FF[|\tilde{\eta}_1^+|^2]|^2 \dk_1 \dk_2+
\underline{O}(\varepsilon^{1+2\theta}\nn \eta_1 \nn^4), \\
\int_{\R^2} \frac{k_1^2}{(1-\Lambda)k_1^2+k_2^2}  |\FF[|\eta_1^+|^2]|^2\dk_1\dk_2&=  \int_{\R^2}  \frac{k_1^2}{(1-\Lambda)k_1^2+k_2^2}  |\FF[|\tilde{\eta}_1^+|^2]|^2\dk_1\dk_2+\underline{O}(\varepsilon^{1+2\theta}\nn \eta_1 \nn^4)
\end{align*}
(note that the multipliers $k_1^2/|k|^2$ and $k_1^2/((1-\Lambda)k_1^2+k_2^2)$ are bounded).
\end{proof}

Finally, we write
$$\tilde{\eta}_1^+(x,z) = \frac{1}{2}\varepsilon \zeta(\varepsilon x,\varepsilon z) \ee^{\ii \omega x},$$
abbreviating the composite change of variable (an isomorphism $\chi(\varepsilon D)H^1(\R) \rightarrow \chi(D)H^1(\R)$)
and its inverse to $\eta_1(\zeta)$ and $\zeta(\eta_1)$,
and define
$$\TT_\varepsilon(\zeta):=\varepsilon^{-2}\widetilde{\JJ}_\varepsilon(\eta_1(\zeta)).$$
Using Lemma \ref{lem:cov to DS}, one finds that
\begin{equation}
\TT_\varepsilon(\zeta)=\QQ(\zeta)-\mathcal{S}(\zeta)+\varepsilon^{3\theta-2}\RR_\varepsilon(\zeta),
\label{Final red func}
\end{equation}
where
\begin{align*}
\QQ(\zeta)&=\int_{\R^2}\left(a_1|\zeta_x|^2+a_2|\zeta_z|^2+a_3|\zeta|^2\right)\dx\dz, \\
\mathcal{S}(\zeta)&= C_1\int_{\R^2}\frac{k_1^2}{(1-\Lambda)k_1^2+k_2^2}
|\FF[|\zeta|^2]|^2\dk_1\dk_2
+C_2\int_{\R^2}|\zeta|^4\dx\dz,
\end{align*}
$$a_1=\frac{1}{8}\partial_{k_1}^2\tilde{g}(\omega,0), \qquad a_2=\frac{1}{8}\partial_{k_2}^2\tilde{g}(\omega,0), \qquad a_3=\frac{1}{4}\Lambda f(\omega)$$
and
$\RR_\varepsilon(\zeta)=\underline{\OO}(\|\zeta\|_1^2)$ (note that
$\nn\tilde{\eta}_1\nn^2=\frac{1}{2} \|\zeta\|^2_1$). It is convenient to choose the concrete value $\theta =\frac{5}{6}$,
so that $\varepsilon^{3\theta-2}=\varepsilon^{1/2}$.
We study the functional $\TT_\varepsilon$ in
$$U_\varepsilon:=B_R(0)\subseteq H_\varepsilon^1(\R^2):=\chi(\varepsilon D)H^1(\R^2),$$
where $R$ is independent of $\varepsilon$ and satisfies $R^2 \leq 2R_1^2 \sup \tilde{g}/\tilde{g}_2$; we may therefore
take it arbitrarily large.

\begin{remark} \label{rem:crit pts relation}
By construction 
$$\mathrm{d}\JJ_\varepsilon[\eta_1+\eta_2(\eta_1)](\rho_1) = \varepsilon^2 \mathrm{d}\TT_\varepsilon [\zeta(\eta_1)](\zeta(\rho_1)), \qquad \rho_1 \in \XX_1,$$
for each $\eta_1 \in X_1$ and
$$
\mathrm{d}\TT_\varepsilon [\zeta](\xi) = \varepsilon^{-2}\mathrm{d}\JJ_\varepsilon[\eta_1(\zeta)+\eta_2(\eta_1(\zeta))](\eta_1(\xi)), \qquad \xi \in H_\varepsilon^1(\R^2),
$$
for each $\zeta \in U_\varepsilon$.
\end{remark}

\section{Existence theory} \label{sec:Existence theory}

\subsection{The natural constraint}
We find critical points of $\TT_\varepsilon$ by minimising it over its
\emph{natural constraint set}
$$N_\varepsilon=\{\zeta\in U_\varepsilon:\zeta\neq 0,\mathrm{d}\TT_\varepsilon[\zeta](\zeta)=0\},$$
noting the identity
\begin{equation}
\label{ relation Q-S}
0=\mathrm{d}\TT_\varepsilon[\zeta](\zeta)=2\QQ(\zeta)-4\mathcal{S}(\zeta)
+\varepsilon^{1/2}\mathrm{d}\RR_\varepsilon[\zeta](\zeta)
\end{equation}
and resulting estimate
\begin{align}
\mathrm{d}^2\TT_\varepsilon[\zeta](\zeta,\zeta)
& = 2\QQ(\zeta)-12\mathcal{S}(\zeta)+\varepsilon^{1/2}\mathrm{d}^2\RR_\varepsilon[\zeta](\zeta,\zeta) \nonumber \\
& = -4\QQ(\zeta) -3 \varepsilon^{1/2}\mathrm{d}\RR_\varepsilon[\zeta](\zeta) + \varepsilon^{1/2}\mathrm{d}^2\RR_\varepsilon[\zeta](\zeta,\zeta) \nonumber \\
& =
-4\QQ(\zeta)+O(\varepsilon^{1/2} \|\zeta\|^2_1) \label{ Second derivative}
\end{align}
for points $\zeta \in N_\varepsilon$.

\begin{remark} \label{rem:ground states}
Any `ground state', that is a minimiser $\zeta^\star$
of $\TT_\varepsilon$ over $N_\varepsilon$, is a (necessarily nonzero) critical point of $\TT_\varepsilon$.
Define $\GG_\varepsilon: U_\varepsilon\setminus \{0\} \rightarrow \R$ by $\GG_\varepsilon(\zeta)=\mathrm{d}\TT_\varepsilon[\zeta](\zeta)$, so that $N_\varepsilon = \GG_\varepsilon^{-1}(0)$ and $\mathrm{d}\GG_\varepsilon[\zeta]$
does not vanish on $N_\varepsilon$ (since $\mathrm{d}\GG_\varepsilon[\zeta](\zeta)=\mathrm{d}^2\TT_\varepsilon[\zeta](\zeta,\zeta)<0$ for $\zeta \in N_\varepsilon$).
There exists a Lagrange multiplier $\mu$ such that
$$\mathrm{d}\TT_\varepsilon[\zeta^\star]-\mu\mathrm{d}\GG_\varepsilon[\zeta^\star]=0,$$
and the calculation
$$\mu=\frac{(\mathrm{d}\TT_\varepsilon[\zeta^\star]-\mu\mathrm{d}\GG_\varepsilon[\zeta^\star])(\zeta^\star)}
{\mathrm{d}\GG_\varepsilon[\zeta^\star](\zeta^\star)}
=\frac{(\mathrm{d}\TT_\varepsilon[\zeta^\star]-\mu\mathrm{d}\GG_\varepsilon[\zeta^\star])(\zeta^\star)}
{\mathrm{d}^2\TT_\varepsilon[\zeta^\star](\zeta^\star,\zeta^\star)}$$
shows that $\mu=0$.
\end{remark}

We first present a geometrical interpretation of $N_\varepsilon$ (see Figure \ref{ncs geometry}).

\begin{proposition} \label{prop:nc interpretation}
Any ray in $B_R(0)\setminus \{0\} \subset H_\varepsilon^1(\R^2)$ intersects
$N_\varepsilon$ in at most one point and the value of $\TT_\varepsilon$ along such a ray attains a strict maximum at this point.
\end{proposition}
\begin{proof}
Let $\zeta\in B_R(0)\setminus \{0\} \subset H_\varepsilon^1(\R^2)$ and consider the value of $\TT_\varepsilon$
along the ray in $B_R(0)\setminus \{0\}$ through $\zeta$, that is the set
$\{\lambda\zeta:0<\lambda<R/\|\zeta\|_1\} \subset H_\varepsilon^1(\R^2)$.
The calculation
$$\frac{\mathrm{d}}{\mathrm{d}\lambda} \TT_\varepsilon(\lambda \zeta)=\mathrm{d}\TT_\varepsilon[\lambda\zeta](\zeta)=\lambda^{-1}\mathrm{d}\TT_\varepsilon[\lambda\zeta](\lambda\zeta)$$
shows that $\frac{\mathrm{d}}{\mathrm{d}\lambda} \TT_\varepsilon(\lambda \zeta)=0$ if and only if
$\lambda \zeta \in N_\varepsilon$; furthermore
\begin{align*}
\frac{\mathrm{d}^2}{\mathrm{d}\lambda^2}\TT_\varepsilon(\lambda\zeta)
&=
2\QQ(\zeta)-12\lambda^{-2} S(\lambda \zeta) +
\varepsilon^{1/2}\mathrm{d}^2\RR_\varepsilon[\lambda\zeta](\zeta,\zeta)\\
&=
-4\QQ(\zeta)
-3\lambda^{-2}\varepsilon^{1/2}\mathrm{d}\RR_\varepsilon[\lambda\zeta](\lambda\zeta)
+\varepsilon^{1/2}\mathrm{d}^2\RR_\varepsilon[\lambda\zeta](\zeta,\zeta)
\\
&=
-4\QQ(\zeta)+O(\varepsilon^{1/2} \|\zeta\|^2_1)\\
&<0
\end{align*}
for each $\zeta$ with $\lambda \zeta \in N_\varepsilon$.
\end{proof} 

\begin{remark}
Proposition \ref{prop:nc interpretation} also holds for $\TT_0: H^1(\R^2) \rightarrow \R$
(with $R=\infty$); in this case every ray intersects $N_0$ in precisely one point.
\end{remark}

Using \eqref{ relation Q-S}, we can eliminate respectively $\mathcal{S}(\zeta)$ and $\QQ(\zeta)$
from \eqref{Final red func} to obtain formulae
\begin{equation}
\TT_\varepsilon(\zeta)
\label{ J=Q}
=\frac 1 2 \QQ(\zeta)+
\varepsilon^{1/2}\left(\RR_\varepsilon(\zeta)-\frac{1}{4}\mathrm{d}\RR_\varepsilon[\zeta](\zeta)\right)
\end{equation}
and
\begin{equation}
\label{ J=S}
\TT_\varepsilon(\zeta)=
\mathcal{S}(\zeta)
+\varepsilon^{1/2}\left(\RR_\varepsilon(\zeta)
-\frac{1}{2}\mathrm{d}\RR_\varepsilon[\zeta](\zeta)\right)
\end{equation}
for $\zeta \in N_\varepsilon$ which lead to
\emph{a priori} bounds for $\TT_\varepsilon|_{N_\varepsilon}$.

\begin{proposition} \label{prop:Bded from below}
There exist constants $D_1$, $D_2>0$ such that
$$\QQ(\zeta) \geq D_1\| \zeta \|_1^2$$
for all $\zeta \in H^1(\R^2)$ and
$$\TT_\varepsilon(\zeta) \geq \tfrac{1}{4} D_1 \|\zeta\|_1^2, \qquad \|\zeta\|_1 \geq D_2$$
for each $\zeta \in N_\varepsilon$.
Furthermore, each
$\zeta\in N_\varepsilon$ with
$\TT_\varepsilon(\zeta)<\frac{1}{4}D_1(R-1)^2$ satisfies
$\|\zeta\|_1<R-1$.
\end{proposition}

\begin{proof}
The first estimate is obtained by choosing $D_1 = \min(a_1,a_2,a_3)$. Let $\zeta \in N_\varepsilon$. Using \eqref{ J=Q}, we find that
$$
\TT_\varepsilon(\zeta)
\geq \tfrac 1 4 \QQ(\zeta)\\
\geq \tfrac 1 4 D_1\|\zeta\|_1^2,
$$
so that in particular $\TT_\varepsilon(\zeta)<\frac{1}{4}D_1(R-1)^2$ implies that $\|\zeta\|_1<R-1$. The lower bound
for $\|\zeta_1\|_1$ follows from the estimate
$$\|\zeta\|_1^2 \lesssim 2 \QQ(\zeta) \leq 4\mathcal{S}(\zeta) + \varepsilon^{1/2}\|\mathrm{d}\RR_\varepsilon(\zeta)\|\|\zeta\|_1
\lesssim \|\zeta\|_1^4 + \varepsilon^{1/2}\|\zeta\|_1^2,$$
in which we have used \eqref{ relation Q-S}.
\end{proof}

\begin{remark} \label{rem:inf is positive}
It follows from Proposition \ref{prop:Bded from below} that the quantity $c_\varepsilon := \inf_{N_\varepsilon} \TT_\varepsilon$ satisfies\linebreak 
$\liminf_{\varepsilon \rightarrow 0} c_\varepsilon>0$.
\end{remark}

\begin{lemma}
For each sufficiently large value of $R$ (chosen independently of $\varepsilon$) there exists\linebreak
$\zeta^\star \in N_\varepsilon$ such that
$\TT_\varepsilon(\zeta^\star) < \frac{1}{4}D_1(R-1)^2$.
\end{lemma}
\begin{proof}
Choose $\zeta_0\in H^1(\R^2)\setminus\{0\}$ and $R>1$ such that
$$\frac{\QQ(\zeta_0)^2}{\mathcal{S}(\zeta_0)}<D_1(R-1)^2.$$
The calculation
$$\mathrm{d}\TT_0[\lambda_0 \zeta_0](\lambda_0\zeta_0)
= 2\lambda_0^2 \QQ(\zeta_0) - 4\lambda_0^4 \mathcal{S}(\zeta_0)
$$
shows that $\lambda_0\zeta_0 \in N_0$, where
$$\lambda_0=\left(\frac{\QQ(\zeta_0)}{2\mathcal{S}(\zeta_0)}\right)^{1/2}.$$
It follows that $\lambda_0\zeta_0$ is the unique point on its ray which lies on $N_0$, and
\begin{equation}
\frac{\mathrm{d}}{\mathrm{d} \lambda} \TT_0(\lambda \zeta_0) \Big|_{\lambda=\lambda_0} = 0,
\qquad
\frac{\mathrm{d}^2}{\mathrm{d} \lambda^2} \TT_0(\lambda \zeta_0) \Big|_{\lambda=\lambda_0} < 0.
\label{ zero ray}
\end{equation}
Furthermore
\begin{equation}
\TT_0(\lambda_0\zeta_0)=\tfrac{1}{2}\QQ(\lambda_0\zeta_0)=\frac{\QQ(\zeta_0)^2}{4\mathcal{S}(\zeta_0)}<\tfrac{1}{4}D_1(R-1)^2, \label{ Value of J0}
\end{equation}
so that
$$\|\lambda_0 \zeta_0\|_1 < R-1.$$

Let $\zeta_\varepsilon = \chi(\varepsilon D)\zeta_0$, so that
$\zeta_\varepsilon\in H^1_\varepsilon(\R^2)\subset H^1(\R^2)$ with
$\lim_{\varepsilon\rightarrow 0}\|\zeta_\varepsilon-\zeta_0\|_1=0$, and in particular
$$\|\lambda_0\zeta_\varepsilon\|_1 < R-1.$$
According to \eqref{ zero ray} we can find $\widetilde \gamma>1$ such that $\tilde{\gamma} \|\lambda_0 \zeta_\varepsilon \|_1 < R$
(so that $\tilde{\gamma}\lambda_0\zeta_\varepsilon \in U_\varepsilon$) and
$$
\frac{\mathrm{d}}{\mathrm{d} \lambda} \TT_0(\lambda \zeta_0) \Big|_{\lambda=\tilde{\gamma}^{-1}\lambda_0}>0,
\qquad
\frac{\mathrm{d}}{\mathrm{d} \lambda} \TT_0(\lambda \zeta_0) \Big|_{\lambda=\tilde{\gamma}\lambda_0}<0,
$$
and therefore
$$
\frac{\mathrm{d}}{\mathrm{d} \lambda} \TT_\varepsilon(\lambda \zeta_\varepsilon) \Big|_{\lambda=\tilde{\gamma}^{-1}\lambda_0}>0,
\qquad
\frac{\mathrm{d}}{\mathrm{d} \lambda} \TT_\varepsilon(\lambda \zeta_\varepsilon) \Big|_{\lambda=\tilde{\gamma}\lambda_0}<0
$$
(the quantities on the left-hand sides of the inequalities on the
second line converge to those on the first as $\varepsilon \rightarrow 0$).
It follows that there exists $\lambda_\varepsilon\in (\widetilde \gamma^{-1}\lambda_0,\widetilde \gamma \lambda_0)$
with
$$\frac{\mathrm{d}}{\mathrm{d} \lambda} \TT_\varepsilon(\lambda \zeta_\varepsilon) \Big|_{\lambda=\lambda_\varepsilon}=0,$$
that is $\lambda_\varepsilon \zeta_\varepsilon \in N_\varepsilon$, and we conclude that
this value of $\lambda_\varepsilon$ is unique (see Proposition \ref{prop:nc interpretation})
and that $\lim_{\varepsilon \rightarrow 0} \lambda_\varepsilon = \lambda_0$.
Using the limit
$$\lim_{\varepsilon\rightarrow 0}
\TT_\varepsilon(\lambda_\varepsilon\zeta_\varepsilon)
=\TT_0(\lambda_0\zeta_0)$$
and \eqref{ Value of J0}, we find that
\[\TT_\varepsilon(\lambda_\varepsilon\zeta_\varepsilon)<\tfrac{1}{4}D_1(R-1)^2. \qedhere\]
\end{proof}

\begin{corollary}
Any minimising sequence $\{\zeta_n\}$ of $\TT_\varepsilon|_{N_{\varepsilon}}$ satisfies
$$
\limsup_{n\rightarrow\infty} \|\zeta_n\|_1 <R-1.
$$
\end{corollary}

Our final result shows that there is a minimising sequence for $\TT_\varepsilon|_{N_\varepsilon}$
which is also a Palais-Smale sequence.

\begin{theorem}
\label{thm:Existence of a good minimising sequence}
There exists 
a minimising sequence $\{\zeta_n\} \subset B_{R-1}(0)$
for $\TT_\varepsilon|_{N_\varepsilon}$
with
$\mathrm{d}\TT_\varepsilon[\zeta_n]\rightarrow 0$ as $n \rightarrow \infty$.
\end{theorem}

\begin{proof} Ekeland's variational principle (Ekeland \cite[Theorem 3.1]{Ekeland74}) asserts the
existence of a minimising sequence $\{\zeta_n\}$ for $\TT_\varepsilon|_{N_\varepsilon}$ and a sequence
$\{\mu_n\}$ of real numbers such that
$$\|\mathrm{d}\TT_\varepsilon[\zeta_n]-\mu_n\mathrm{d}\GG_\varepsilon[\zeta_n]\| \rightarrow 0$$
as $n \rightarrow \infty$, and the calculation
$$\mu_n=\frac{-(\mathrm{d}\TT_\varepsilon[\zeta_n]-\mu_n \mathrm{d}\GG_\varepsilon[\zeta_n])(\zeta_n)}
{\mathrm{d}\GG_\varepsilon[\zeta_n](\zeta_n)}
=\frac{o(\|\zeta_n\|_1)}
{\mathrm{d}^2\TT_\varepsilon[\zeta_n](\zeta_n,\zeta_n)}$$
shows that $\lim_{n \rightarrow \infty} \mu_n = 0$ because $-\mathrm{d}^2\TT_\varepsilon[\zeta_n](\zeta_n,\zeta_n) \gtrsim \|\zeta_n\|_1^2$
and $\|\zeta_n\|_1 \geq D_2$ (see \eqref{ Second derivative} and Proposition \ref{prop:Bded from below}).
\end{proof}

We conclude this section with a remark which applies in particular to the sequence constructed in
Theorem \ref{thm:Existence of a good minimising sequence} (extracting a subsequence if necessary, we
always assume that such sequences are weakly convergent).

\begin{remark} \label{rem:defn of wc}
Suppose that $\varepsilon>0$, so that $H_\varepsilon^1(\R^2)$ is topologically identical to\linebreak $H_\varepsilon^s(\R^2):=\chi(\varepsilon D)H^s(\R^2)$ for all $s \geq 0$.
Any sequence which is weakly convergent in $H_\varepsilon^1(\R^2)$ is therefore in particular also weakly convergent in
$H^3(\R^2)$ and strongly convergent in $H_\mathrm{loc}^3(\R^2)$. We thus henceforth use the phrase `weakly convergent'
synonymously with `weakly convergent in $H^3(\R^2)$ and strongly convergent in $H_\mathrm{loc}^3(\R^2)$' when
discussing sequences in $H_\varepsilon^1(\R^2)$ for $\varepsilon>0$; all other convergence properties of such sequences are deduced
using standard embedding theorems.
\end{remark}

\subsection{Existence of a critical point} \label{Existence 1}

In this section we fix $\varepsilon>0$ and show that the minimising sequence for $\TT_\varepsilon|_{N_\varepsilon}$
constructed in Theorem \ref{thm:Existence of a good minimising
sequence} converges weakly (up to translations) to a nontrivial critical point of $\TT_\varepsilon$.
The result is stated in Theorem \ref{thm:Existence of a critical point} below; the following lemmata, which
show respectively that Palais-Smale sequences converge weakly to critical points, and that `vanishing'
does not occur, are used in its proof.

\begin{lemma} \label{lem:weak limit of ps sequence}
Suppose that $\{\zeta_n\}$ is a sequence in $B_{R-1}(0) \subset H_\varepsilon^1(\R^2)$ with the property that
$\mathrm{d}\TT_\varepsilon[\zeta_n]\rightarrow 0$ as $n \rightarrow \infty$. Its weak limit
$\zeta_\infty$ is a critical point
of $\TT_\varepsilon$ and $\eta_n = \eta_1(\zeta_n)+\eta_2(\eta_1(\zeta_n))$ converges weakly
in $H^3(\R^2)$ to $\eta_\infty = \eta_1(\zeta_\infty) + \eta_2(\eta_1(\zeta_\infty))$ (which is a critical point of
$\JJ_\varepsilon$).
\end{lemma}
\begin{proof}
Observe that
$$| \mathrm{d}\JJ_\varepsilon[\eta_{1,n}+\eta_2(\eta_{1,n})](\rho_1) | \lesssim \varepsilon^2 \|\mathrm{d}\TT_\varepsilon [\zeta_n]\| \nn \rho_1 \nn$$
for each $\rho_1 \in \XX_1$, where we have abbreviated $\{\eta_1(\zeta_n)\}$ to $\{\eta_{1,n}\}$ (see Remark
\ref{rem:crit pts relation}), so that
$$
\mathrm{d}\JJ_\varepsilon[\eta_{1,n}+\eta_2(\eta_{1,n})]\rightarrow 0
$$
and hence
\begin{equation}
\langle \JJ_\varepsilon^\prime(\eta_{1,n}+\eta_2(\eta_{1,n})), \rho \rangle_0 \rightarrow 0
\label{ J' goes to zero}
\end{equation}
for each $\rho \in H^3(\R^2)$
as $n \rightarrow \infty$.

The sequence $\{\eta_{1,n}\} \subset X_1$ converges weakly in $\XX_1$ to
$\eta_{1,\infty}=\eta_1(\zeta_\infty) \in X_1$, and it follows from Remark \ref{rem:Gradients wc}
that $\{F(\eta_{1,n})\}$ converges weakly in $H^3(\R^2)$ to $F(\eta_{1,\infty})$. Let $\eta_{3,n}$ be the
unique solution in $X_3$ of equation \eqref{eta3 eqn} with $\eta_1=\eta_{1,n}$, so that
$$\eta_{3,n}=G(\eta_{1,n},\eta_{3,n}).$$
Observing that $G: X_1 \times X_3 \rightarrow H^3(\R^2)$ is weakly continuous, we find that
the weak limit $\eta_{3,\infty} \in X_3$ of $\{\eta_{3,n}\}$ in $H^3(\R^2)$ satisfies
$$\eta_{3,\infty}=G(\eta_{1,\infty},\eta_{3,\infty}),$$
so that $\eta_{3,\infty} = \eta_3(\eta_{1,\infty})$ (because the fixed-point equation
$\eta_3 = G(\eta_{1,\infty}, \eta_3)$ has a unique solution in $X_3$).

Altogether this argument shows that
$\{\eta_2(\eta_{1,n})\}$ converges weakly in $\XX_2$ to $\eta_{2,\infty}=\eta_2(\eta_{1,\infty})$.
Writing
$\eta_\infty=\eta_{1,\infty}+\eta_{2,\infty}$
and using \eqref{ J' goes to zero} and Corollary \ref{cor:J' is wc}, we find that
$$\langle \JJ_\varepsilon^\prime(\eta_\infty),\rho \rangle_0=0$$
for all $\rho \in H^3(\R^2)$. Consequently $\eta_{1,\infty}$ is a critical point of $\widetilde{\JJ}_\varepsilon$
(see the remarks at the end of Section \ref{section: Reduction}) and $\zeta_\infty$ is a 
critical point of $\TT_\varepsilon$.
\end{proof}
\begin{lemma} \label{lem:no vanishing}
Every sequence $\{\zeta_n\}\subset N_\varepsilon$ has the property that
$$
\lim_{n\rightarrow \infty}\sup_{j\in\Z^2}\|
\zeta_n
\|_{H^{3}(\{w:|w-j|_\infty<1/2\})} \neq 0.
$$
\end{lemma}
\begin{proof}
The contrary assumption
$$
\lim_{n\rightarrow \infty}\sup_{j\in\Z^2}\|
\zeta_n
\|_{H^{3}(\{w:|w-j|_\infty<1/2\})}=0
$$
implies that
$$\lim_{n\rightarrow \infty}\|\zeta_n\|_\infty=0$$
and hence that
$$
\mathcal{S}(\zeta_n) \lesssim \int_{\R^2}|\zeta_n|^4 \dx\dz
\leq \|\zeta_n\|_\infty^2 R^2
\rightarrow 0
$$
as $n \rightarrow \infty$. It follows from \eqref{ J=S} that
$$\lim_{n\rightarrow \infty}\TT_\varepsilon(\zeta_n) = O(\varepsilon^{1/2}),$$
which contradicts Proposition \ref{prop:Bded from below}.
\end{proof}
\begin{theorem}
\label{thm:Existence of a critical point}
Let $\{\zeta_n\} \subset B_{R-1}(0)$ be a minimising sequence for $\TT_\varepsilon|_{N_\varepsilon}$
with
$\mathrm{d}\TT_\varepsilon[\zeta_n]\rightarrow 0$ as $n \rightarrow \infty$.
There exists a sequence $\{w_n\} \subset \Z^2$ with the property that
$\{\zeta_n(\cdot+w_n)\}$
converges weakly to a nontrivial critical point $\zeta_\infty$ of
$\TT_\varepsilon$. The function $\eta_\infty = \eta_1(\zeta_\infty) + \eta_2(\eta_1(\zeta_\infty))$
is a nonzero critical point of $\JJ_\varepsilon$.
\end{theorem}
\begin{proof}
In view of Lemma \ref{lem:weak limit of ps sequence} it remains
only to demonstrate that we can
select $\{w_n\} \subset \Z^2$ so that the weak limit of $\{\zeta_n(\cdot+w_n)\}$ is not zero. Supposing the
contrary, we find that
$$\lim_{n \rightarrow \infty}
\|\zeta_n(\cdot+w_n) \|_{H^3(\{w:|w|_\infty<1/2\})}=0
$$
for all sequences $\{w_n\} \subset \Z^2$ (see Remark \ref{rem:defn of wc})
and hence
$$
\lim_{n\rightarrow \infty}\sup_{j\in\Z^2}\|
\zeta_n
\|_{H^{3}(\{w:|w-j|_\infty<1/2\})} = 0,
$$
which
contradicts Lemma \ref{lem:no vanishing}.
\end{proof}

\subsection{Existence of a ground state} \label{Existence 2}

In this section we improve the result of Theorem \ref{thm:Existence of a critical point} (again fixing $\varepsilon>0$)
by showing that we
can choose the sequence $\{w_n\}$ to ensure convergence to a ground state.
\begin{theorem} \label{thm:Ground state}
Let $\{\zeta_n\} \subset B_{R-1}(0)$ be a minimising sequence for $\TT_\varepsilon|_{N_\varepsilon}$
with
$\mathrm{d}\TT_\varepsilon[\zeta_n]\rightarrow 0$ as $n \rightarrow \infty$.
There exists a sequence $\{w_n\} \subset \Z^2$ with the property that a subsequence of
$\{\zeta_n(\cdot+w_n)\}$
converges weakly to a ground state $\zeta_\infty$ (so that $\TT_\varepsilon(\zeta_\infty)=c_\varepsilon$).

Moreover, the sequence $\{\eta_n(\cdot+w_n)\} \subset U$, where $\eta_n = \eta_1(\zeta_n)+\eta_2(\eta_1(\zeta_n))$, $n \in \N$,
converges strongly in $H^s({\mathbb R}^2)$ for $s \in [0,3)$ to
$\eta_\infty = \eta_1(\zeta_\infty) + \eta_2(\eta_1(\zeta_\infty))$, and this function
is a nonzero critical point of $\JJ_\varepsilon$.
\end{theorem}

The proof of Theorem \ref{thm:Ground state} consists of Lemma \ref{lem:Step 1},
Proposition \ref{prop:Step 2} and Lemmata \ref{lem:Step 3}, \ref{lem:Step 4} below; here $\{\zeta_n\} \subset B_{R-1}(0)$
is a minimising sequence for $\TT_\varepsilon|_{N_\varepsilon}$ with
$\mathrm{d}\TT_\varepsilon[\zeta_n]\rightarrow 0$ as $n \rightarrow \infty$.

\begin{lemma} \label{lem:Step 1}
There exists $\{w_n\} \subset \Z^2$ and $\zeta_\infty \in N_\varepsilon$ such that
$\zeta_n(\cdot+w_n) \rightharpoonup \zeta_\infty$ and
$$
\lim_{n\rightarrow \infty}\sup_{j\in\Z^2}\left\|\zeta_n(\cdot+w_n)-\zeta_\infty
\right\|_{H^{3}(\{w:|w-j|_\infty<1/2\})}
=0.
$$
\end{lemma}
\begin{proof}
This lemma is established by applying the abstract concentration-compactness theory given in
the Appendix  and showing that `concentration' occurs.
We set $H=H^3((-\frac{1}{2},\frac{1}{2})^2)$, 
define $x_n\in \ell^2(\Z^2,H)$ for $n \in \N$ by
$$x_{n,j}=\zeta_n(\cdot+j)|_{(-\frac{1}{2},\frac{1}{2})^2}
\in H^3((-\tfrac{1}{2},\tfrac{1}{2})^2), \qquad j\in \Z^2,$$
and apply Lemmata \ref{lem:cc1} and \ref{lem:cc2} to the sequence $\{x_n\}\subset \ell^2(\Z^2,H)$,
noting that
$$\|x_n\|_{\ell^2(\Z^2,H)}=\|\zeta_n\|_3, \qquad
\|x_n\|_{\ell^\infty(\Z^2,H)}=\sup_{k\in\Z^2}\|
\zeta_n
\|_{H^{3}(\{w:|w-k|_\infty<1/2\})}$$
for $n \in \N$. Assumptions (i) and (ii) follow from the fact that 
$\{\zeta_n(\cdot+j)|_{(-\frac{1}{2},\frac{1}{2})^2}:n\geq 1,j\in \Z^2\}$
is bounded in $H^s((-\frac{1}{2},\frac{1}{2})^2)$ for all $s \geq 0$, while
assumption (iii) is verified by Lemma \ref{lem:no vanishing}. Given $\varepsilon>0$, the theory
asserts the existence of a natural number $m$, sequences
$\{w_n^1\}, \ldots, \{w_n^m\} \subset \Z^2$ with
\begin{equation}
\lim_{n \rightarrow \infty} |w_n^{m^{\prime\prime}}-w_n^{m^\prime}| \rightarrow \infty, \qquad 1 \leq m^{\prime\prime} < m^\prime \leq m,
\label{ Split}
\end{equation}
and functions
$\zeta^1,\ldots,\zeta^m\in  H_\varepsilon^1(\R^2)\setminus\{0\}$ such that
$\zeta_n(\cdot+w^{m^\prime}_n) \rightharpoonup \zeta^{m^\prime}$
(see Remark \ref{rem:defn of wc}),
\begin{equation}
\label{ approach}
\limsup_{n\rightarrow \infty}\sup_{j\in\Z^2}\left\|
\zeta_n-\sum_{\ell=1}^m\zeta^\ell(\cdot-w^\ell_n)
\right\|_{H^{3}(\{w:|w-j|_\infty<1/2\})}
\leq \varepsilon,
\end{equation}
\begin{equation}\label{ sum phil}
\sum_{\ell=1}^m\|\zeta^\ell\|_3^2\leq
\limsup_{n\rightarrow \infty}\|\zeta_n\|_3^2
\end{equation}
and
\begin{equation}
\label{ concentrate}
\lim_{n\rightarrow \infty}\sup_{j\in\Z^2}\left\|\zeta_n-\zeta^1(\cdot-w^1_n)
\right\|_{H^{3}(\{w:|w-j|_\infty<1/2\})}
=0
\end{equation}
if $m=1$. It follows from Lemma \ref{lem:weak limit of ps sequence}
that $\mathrm{d}\TT_\varepsilon[\zeta^\ell]=0$, so that
$\zeta^\ell \in N_\varepsilon$ and $\TT_\varepsilon(\zeta^\ell) \geq c_\varepsilon$.

Writing
$$\tilde{\zeta}_n=\sum_{\ell=1}^m\zeta^\ell(\cdot-w^\ell_n)$$
and
$\mathcal{S}(f)=[|f|^2,|f|^2]$, where
$$[f_1,f_2]:=
C_1\int_{\R^2}\frac{k_1^2}{(1-\Lambda)k_1^2+k_2^2}
\FF[f_1]\overline{\FF[f_2]}\dk_1\dk_2
+C_2\int_{\R^2}f_1 f_2\dx\dz
$$
defines a continuous inner product for $L^2(\R^2)$, we find that
\begin{align}
\limsup_{n\rightarrow \infty}\hspace{-5mm}&\hspace{5mm}\left|\mathcal{S}(\zeta_n)-\mathcal{S}(\tilde{\zeta}_n)\right|
\nonumber \\
&=
\limsup_{n\rightarrow \infty}\left[\,|\zeta_n|^2-|\tilde{\zeta}_n|^2,|\zeta_n|^2+|\tilde{\zeta}_n|^2\,\right]
\nonumber \\
& \lesssim \limsup_{n\rightarrow \infty}\left\|\, |\zeta_n|^2-|\tilde{\zeta}_n|^2\right\|_0
\left\|\, |\zeta_n|^2+|\tilde{\zeta}_n|^2\right\|_0
\nonumber \\
& = \limsup_{n\rightarrow \infty}\left\{\sum_{j\in \Z^2}\left\|\mathrm{Re}\left( (\zeta_n-\tilde{\zeta}_n)
\overline{(\zeta_n+\tilde{\zeta}_n)}\right)\right\|_{L^2_j}^2\right\}^{1/2}\hspace{-3mm}\left\|\, |\zeta_n|^2+|\tilde{\zeta}_n|^2\right\|_0
\nonumber \\
& \leq\limsup_{n\rightarrow \infty}\left\{\sum_{j\in \Z^2}
\left\| \zeta_n-\tilde{\zeta}_n\right\|_{L^4_j}^2
\left\| \zeta_n+\tilde{\zeta}_n \right\|_{L^4_j}^2
\right\}^{1/2}\hspace{-3mm} \left\|\, |\zeta_n|^2+|\tilde{\zeta}_n|^2\right\|_0\nonumber \\
& \lesssim \limsup_{n\rightarrow \infty}\left\{\sum_{j\in \Z^2}
\left\| \zeta_n-\tilde{\zeta}_n\right\|_{H^3_j}^2
\left\| \zeta_n+\tilde{\zeta}_n \right\|_{H^3_j}^2
\right\}^{1/2}\hspace{-3mm} \left\|\, |\zeta_n|^2+|\tilde{\zeta}_n|^2\right\|_0
\nonumber \\
& \lesssim \varepsilon\,  \limsup_{n\rightarrow \infty}\left\{\sum_{j\in \Z^2}
\left\| \zeta_n+\tilde{\zeta}_n \right\|_{H^3_j}^2
\right\}^{1/2}\hspace{-3mm}
\left\|\, |\zeta_n|^2+|\tilde{\zeta}_n|^2\right\|_0
\nonumber \\
& \lesssim \varepsilon \limsup_{n\rightarrow \infty}
\left\| \zeta_n+\tilde{\zeta}_n \right\|_3
\left\|\, |\zeta_n|^2+|\tilde{\zeta}_n|^2\right\|_0
\nonumber \\
& \lesssim \varepsilon \limsup_{n\rightarrow \infty}
\left(\left\| \zeta_n \right\|_3+
\| \tilde{\zeta}_n \|_3\right)
\left(\left\|\zeta_n\right\|_3^2
+\|\tilde{\zeta}_n\|_3^2\right)
\nonumber \\
&\lesssim \varepsilon \limsup_{n\rightarrow \infty}
\| \zeta_n \|_3^3 \nonumber \\
& = O(\varepsilon), \label{ First S estimate}
\end{align}
in which we have used the notation
$$L^2_j=L^2(\{w:|w-j|_\infty<1/2\}),\qquad
L^4_j=L^4(\{w:|w-j|_\infty<1/2\}),$$
$$H^3_j=H^3(\{w:|w-j|_\infty<1/2\})$$
and \eqref{ approach}, \eqref{ sum phil}, and
\begin{align}
\lim_{n\rightarrow \infty } \mathcal{S}(\tilde{\zeta}_n)
&=\lim_{n\rightarrow \infty }\sum_{\ell_1,\ell_2,\ell_3,\ell_4} \left[\zeta^{\ell_1}(\cdot-w_n^{\ell_1})
\overline{\zeta^{\ell_2}(\cdot-w_n^{\ell_2})},
\zeta^{\ell_3}(\cdot-w_n^{\ell_3})\overline{\zeta^{\ell_4}(\cdot-w_n^{\ell_4})}\right]
\nonumber \\
&=\lim_{n\rightarrow \infty }\sum_{\ell_1, \ell_3} \left[\zeta^{\ell_1}(\cdot-w_n^{\ell_1})
\overline{\zeta^{\ell_1}(\cdot-w_n^{\ell_1})},
\zeta^{\ell_3}(\cdot-w_n^{\ell_3})\overline{\zeta^{\ell_3}(\cdot-w_n^{\ell_3})}\right]
\nonumber \\
&=\lim_{n\rightarrow \infty }\sum_{\ell=1}^m \left[\zeta^{\ell}
\overline{\zeta^{\ell}},
\zeta^{\ell}\overline{\zeta^{\ell}}\right]\nonumber \\
&=\sum_{\ell=1}^m \mathcal{S}(\zeta^\ell), \label{ Second S estimate}
\end{align}
where we have used the calculations
\begin{align*}
\lim_{n\rightarrow \infty }\|\zeta^{\ell_1}(\cdot-w_n^{\ell_1})
\overline{\zeta^{\ell_2}(\cdot-w_n^{\ell_2})}\|^2_0
&=\lim_{n\rightarrow \infty }\int_{\R^2}
|\zeta^{\ell_1}(\cdot-w_n^{\ell_1})|^2
\overline{|\zeta^{\ell_2}(\cdot-w_n^{\ell_2})|^2}\dx\dz
\\
&=\lim_{n\rightarrow \infty }\int_{\R^2}
\FF\left[|\zeta^{\ell_1}(\cdot-w_n^{\ell_1})|^2\right]
\overline{\FF\left[|\zeta^{\ell_2}(\cdot-w_n^{\ell_2})|^2\right]}\dk_1\dk_2
\\
&=\lim_{n\rightarrow \infty }\int_{\R^2}
\ee^{\ii(-w_n^{\ell_1}+w_n^{\ell_2})\cdot(k_1,k_2)}
\FF\left[|\zeta^{\ell_1}|^2\right]
\overline{\FF\left[|\zeta^{\ell_2}|^2\right]}\dk_1\dk_2 \\
&=0
\end{align*}
for $\ell_1\neq\ell_2$ and
\begin{eqnarray*}
\lefteqn{
\lim_{n\rightarrow \infty } \left[
\zeta^{\ell_1}(\cdot-w_n^{\ell_1})
\overline{\zeta^{\ell_1}(\cdot-w_n^{\ell_1})},
\zeta^{\ell_3}(\cdot-w_n^{\ell_3})
\overline{\zeta^{\ell_3}(\cdot-w_n^{\ell_3})}
\,\right]}\qquad
\\&=&\lim_{n\rightarrow \infty } \left[
|\zeta^{\ell_1}(\cdot-w_n^{\ell_1})|^2,
|\zeta^{\ell_3}(\cdot-w_n^{\ell_3})|^2
\,\right]
\\&=& C_1\lim_{n\rightarrow \infty } 
\int_{\R^2}\frac{k_1^2}{(1-\Lambda)k_1^2+k_2^2}
\FF\left[|\zeta^{\ell_1}(\cdot-w_n^{\ell_1})|^2\right]
\overline{\FF\left[|\zeta^{\ell_3}(\cdot-w_n^{\ell_3})|^2\right]}\dk_1\dk_2
\\
& & \qquad\mbox{}+C_2\lim_{n\rightarrow \infty } 
\int_{\R^2}|\zeta^{\ell_1}(\cdot-w_n^{\ell_1})|^2
\overline{|\zeta^{\ell_3}(\cdot-w_n^{\ell_3})|^2}\dx\dz
\\ &=& C_1\lim_{n\rightarrow \infty } 
\int_{\R^2}\frac{k_1^2}{(1-\Lambda)k_1^2+k_2^2}
\ee^{\ii(-w_n^{\ell_1}+w_n^{\ell_3})\cdot(k_1,k_2)}
\FF\left[|\zeta^{\ell_1}|^2|\right]
\overline{\FF\left[|\zeta^{\ell_3}|^2\right]}\dk_1\dk_2 \\
& = &0
\end{eqnarray*}
for $\ell_1\neq\ell_3$ (by \eqref{ Split} and the Riemann-Lebesgue lemma).

Using \eqref{ First S estimate}, \eqref{ Second S estimate}, the estimate
$$\mathcal{S}(\zeta^\ell) \geq c_\varepsilon-O(\varepsilon^{1/2})\|\zeta^{\ell}\|^2_{H^1(\R^2)}, \qquad \ell=1,\ldots,m$$
(see \eqref{ J=S}) and \eqref{ sum phil}, we find that
$$\liminf_{n\rightarrow \infty}\mathcal{S}(\zeta_n)
\geq m c_\varepsilon-O(\varepsilon^{1/2}),$$
in which the $O(\varepsilon^{1/2})$ term does not depend on $m$. It follows from \eqref{ J=S} that
$$c_\varepsilon\geq mc_\varepsilon-O(\varepsilon^{1/2}),$$
so that $m=1$ (recall that $\liminf_{\varepsilon\rightarrow 0}c_\varepsilon>0$ (Remark \ref{rem:inf is positive})). The advertised result now follows from
\eqref{ concentrate} with $\zeta_\infty=\zeta^1$ and $w_n=w_n^1$.
\end{proof}

With a slight abuse of notation we now abbreviate the subsequence of $\{\zeta_n(\cdot+w_n)\}$ identified in Lemma
\ref{lem:Step 1} to $\{\zeta_n\}$
and define $\{\eta_n\} \subset U$ by $\eta_n = \eta_1(\zeta_n)+\eta_2(\eta_1(\zeta_n))$, $n \in \N$. The convergence
properties of $\{\eta_n\}$ are examined in Proposition \ref{prop:Step 2}, whose proof makes use of the following remark.

\begin{remark} \label{rem:sup remark}
Suppose that $u_n \rightharpoonup u_\infty$ in $H^3(\R^2)$. The limit
$$
\lim_{n\rightarrow \infty} \|u_n-u_\infty\|_{1,\infty} = 0
$$
holds if and only if $u_n(\cdot - j_n) \rightharpoonup 0$ in $H^3(\R^2)$ for all
unbounded sequences $\{j_n\} \subset \Z^2$.
\end{remark}

\begin{proposition} \label{prop:Step 2}
The sequence $\{\eta_n\}$ converges weakly
in $H^3(\R^2)$ and strongly in $W^{1,\infty}(\R^2)$ (and hence in $W^{1,p}(\R^2)$ for any $p>2$)
to the nonzero critical point
$\eta_\infty = \eta_1(\zeta_\infty) + \eta_2(\eta_1(\zeta_\infty))$ of $\JJ_\varepsilon$.
\end{proposition}
\begin{proof} First note that $\{\eta_n\}$ converges weakly in $H^3(\R^2)$ to $\eta_\infty \neq 0$ and
$\mathrm{d}\JJ_\varepsilon[\eta_\infty]=0$ (see Lemma \ref{lem:weak limit of ps sequence}).

Let $\{j_n\}$ be a sequence in ${\mathbb Z}^2$ with $|j_n| \rightarrow \infty$ as $n \rightarrow \infty$.
It follows from Lemma \ref{lem:Step 1} and Remark \ref{rem:sup remark} that $\zeta_n(\cdot-j_n) \rightharpoonup 0$ in $H^3(\R^2)$
as $n \rightarrow \infty$, and Lemma \ref{lem:weak limit of ps sequence}
shows that $\eta_n(\cdot -j_n) \rightharpoonup 0$ in $H^3(\R^2)$ as $n \rightarrow \infty$, so that 
$$
\lim_{n\rightarrow \infty} \|\eta_n-\eta_\infty\|_{1,\infty} = 0
$$
(Remark \ref{rem:sup remark}).
\end{proof}

\begin{lemma} \label{lem:Step 3}
The sequence $\{\eta_n\}$ satisfies $\JJ_\varepsilon(\eta_n)\rightarrow
\JJ_\varepsilon(\eta_\infty)$ and in particular \linebreak
$
\TT_\varepsilon(\zeta_n)=
\varepsilon^{-2}\JJ_\varepsilon(\eta_n)\rightarrow
\varepsilon^{-2}\JJ_\varepsilon(\eta_\infty)=\TT_\varepsilon(\zeta_\infty)$
as $n \rightarrow \infty$.
\end{lemma}
\begin{proof}
It follows from the relation
$$
\varepsilon^2 \mathrm{d}\TT_\varepsilon [\zeta](\zeta) = \mathrm{d}\JJ_\varepsilon[\eta_1(\zeta)+\eta_2(\eta_1(\zeta))](\eta_1(\zeta))
$$
(see Remark \ref{rem:crit pts relation}) that $\mathrm{d}\JJ_\varepsilon[\eta_n](\eta_n)\rightarrow 0$, and
we demonstrate that
$2\JJ_\varepsilon(\eta_n)-\mathrm{d}\JJ_\varepsilon[\eta_n](\eta_n)\rightarrow
2\JJ_\varepsilon(\eta_\infty)-\mathrm{d}\JJ_\varepsilon[\eta_\infty](\eta_\infty)$
as $n \rightarrow \infty$.

A straightforward calculation yields
$$
2\JJ_\varepsilon(\eta)-\mathrm{d}\JJ_\varepsilon[\eta](\eta)=\GG(\eta)+\frac{(1-\varepsilon^2)\Lambda}{2} \int_{\R^2}\eta (\mathrm{d}K[\eta]\eta)\eta \dx \dz,
$$
where
$$
\GG(\eta)=\beta\int_{\R^2}\frac {(\eta_x^2+\eta_z^2)^2}{(\sqrt{1+\eta_x^2+\eta_z^2}+1)^2
\sqrt{1+\eta_x^2+\eta_z^2}}
\dx \dz.
$$
Observe that $\GG(\eta_n) \rightarrow \GG(\eta_\infty)$ because
$\{\eta_n\}$ converges strongly in $W^{1,\infty}(\R^2)$
and $W^{1,4}(\R^2)$ to $\eta_\infty$. Furthermore
$$ \int_{\R^2}\eta_n (\mathrm{d}K[\eta_n]\eta_n-\mathrm{d}K[\eta_\infty]\eta_\infty)\eta_n \dx \dz\rightarrow 0$$
as $n \rightarrow \infty$ because the map $K(\cdot): W^{1,\infty}(\R^2)
\rightarrow \LL(H^{1/2}(\R^2),H^{-1/2}(\R^2))$ is analytic at the origin, and it follows from
Proposition \ref{prop:Properties of K}(iii) that
$$\int_{\R^2}\eta_n (\mathrm{d}K[\eta_\infty]\eta_\infty)\eta_n \dx \dz\rightarrow
 \int_{\R^2}\eta_\infty (\mathrm{d}K[\eta_\infty]\eta_\infty)\eta_\infty \dx \dz$$
as $n \rightarrow \infty$.
\end{proof}

Finally, we strengthen the convergence result given in Proposition \ref{prop:Step 2}.
\begin{lemma} \label{lem:Step 4}
The sequence $\{\eta_n\}$ converges strongly in $H^s(\R^2)$ for $s \in [0,3)$ to
$\eta_\infty$.
\end{lemma}
\begin{proof}
It suffices to establish this result for $s=1$. Arguing as in the proof of Lemma \ref{lem:Step 3}
and using Proposition \ref{prop:Properties of K}(ii), we find that
\begin{multline*}
2\JJ_\varepsilon(\eta_n)
-\int_{\R^2}(\eta_n^2+\beta\eta_{nx}^2+\beta\eta_{nz}^2) \dx \dz
+(1-\varepsilon^2)\Lambda \int_{\R^2}\eta_n K(0)\eta_n \dx \dz
\\=
-\beta \int_{\R^2}\frac {(\eta_{nx}^2+\eta_{nz}^2)^2}
{(\sqrt{1+\eta_{nx}^2+\eta_{nz}^2}+1)^2}
\, \dx \dz
-(1-\varepsilon^2)\Lambda \int_{\R^2}\eta_n (K(\eta_n)-K(0))
\eta_n \dx \dz
\end{multline*} 
converges to
$$2\JJ_\varepsilon(\eta_\infty)
-\int_{\R^2}(\eta_\infty^2+\beta\eta_{\infty x}^2+\beta\eta_{\infty z}^2) \dx \dz
+(1-\varepsilon^2)\Lambda \int_{\R^2}\eta_\infty K(0)\eta_\infty \dx \dz,$$
so that
\begin{align*}
& \int_{\R^2}(\eta_n^2+\beta\eta_{nx}^2+\beta\eta_{nz}^2) \dx \dz
-(1-\varepsilon^2)\Lambda \int_{\R^2}\eta_n K(0)\eta_n \dx \dz
\\ \rightarrow &
\int_{\R^2}(\eta_\infty^2+\beta\eta_{\infty x}^2+\beta\eta_{\infty z}^2) \dx \dz
-(1-\varepsilon^2)\Lambda \int_{\R^2}\eta_\infty K(0)\eta_\infty \dx \dz
\end{align*}
as $n \rightarrow \infty$, and this result in turn implies that
\begin{align*}
& \int_{\R^2}((\eta_n-\eta_\infty)^2+\beta(\eta_{nx}-\eta_{\infty x})^2
+\beta(\eta_{nz}-\eta_{\infty z})^2) \dx \dz \\
& \qquad\mbox{}
-(1-\varepsilon^2)\Lambda \int_{\R^2}(\eta_n-\eta_\infty) 
K(0)(\eta_n-\eta_\infty) \dx \dz
\\&= \int_{\R^2}(\eta_n^2+\beta\eta_{nx}^2
+\beta \eta_{nz}^2) \dx \dz
-(1-\varepsilon^2)\Lambda \int_{\R^2}\eta_n K(0)\eta_n \dx \dz
\\
& \qquad\mbox{}+\int_{\R^2}(\eta_\infty^2+\beta \eta_{\infty x}^2
+\beta \eta_{\infty z}^2) \dx \dz
-(1-\varepsilon^2)\Lambda \int_{\R^2}\eta_\infty K(0)\eta_\infty \dx \dz \\
& \qquad\mbox{}-2\int_{\R^2}(\eta_n\eta_\infty+\beta \eta_{nx}
\eta_{\infty x}
+\beta \eta_{nz}\eta_{\infty z}) \dx \dz
+2(1-\varepsilon^2)\Lambda \int_{\R^2}\eta_n K(0)\eta_\infty \dx \dz \\
& \rightarrow 0
\end{align*}
as $n \rightarrow \infty$ because $\eta_n$ converges to $\eta_\infty$ weakly in $H^1(\R^2)$.

Noting that
$$\eta \mapsto \left\{\int_{\R^2}(\eta^2+\beta\eta_{x}^2
+\beta\eta_{z}^2) \dx \dz
-(1-\varepsilon^2)\Lambda \int_{\R^2}\eta 
K(0)\eta \dx \dz\right\}^{1/2}$$
defines a norm equivalent to the usual norm for $H^1(\R^2)$, we conclude that $\{\eta_n\}$ converges strongly
in $H^1(\R^2)$ to $\eta_\infty$.
\end{proof}

\appendix

\section*{Appendix: Concentration-compactness} 

\setcounter{equation}{0}
\renewcommand{\theequation}{A.\arabic{equation}}

In this Appendix we establish an abstract result of concentration-compactness type,
following ideas due to Benci \& Cerami \cite{BenciCerami87}.
Consider a sequence $\{x_n\}$ in $\ell^2(\Z^s,H)$, where $H$ is a Hilbert space and $s \in {\mathbb N}$.
Writing $x_n=(x_{n,j})_{j\in \Z^s}$, where  $x_{n,j}\in H$, suppose that
\begin{itemize}
\item[(i)] $\{x_n\}$ is bounded in $\ell^2(\Z^s,H)$,
\item[(ii)] $S=\{x_{n,j}:n \in \N, j \in \Z^s\}$ is relatively compact in $H$,
\item[(iii)] $\limsup_{n\rightarrow \infty}\|x_n\|_{\ell^\infty(\Z^s,H)}>0$.
\end{itemize}

\begin{lemma} \label{lem:cc1}
For each $\varepsilon>0$ the sequence $\{x_n\}$ admits a subsequence
with the following properties.
There exist a finite number $m$ of non-zero vectors $x^1,\ldots,x^m\in \ell^2(\Z^s,H)$ and
 sequences $\{w^1_n\}$, \ldots, $\{w^m_n\}
 \subset \Z^s$ such that
\begin{align}
& T_{-w^{m^\prime}_n}\left(x_n-\sum_{\ell=1}^{m^\prime-1}
T_{w^\ell_n}x^\ell\right)\rightharpoonup x^{m^\prime}, \label{CC1}\\
&
\|x^{m^\prime}\|_{l^\infty(\Z^s,H)}=
\lim_{n\rightarrow \infty}
\left\|x_n-\sum_{\ell=1}^{m^\prime-1}T_{w^\ell_n}x^\ell\right\|_{l^\infty(\Z^s,H)}, \label{CC2} \\
&
\lim_{n\rightarrow \infty}\|x_n\|_{\ell^2(\Z^s,H)}^2=
\sum_{\ell=1}^{m^\prime} \|x^\ell\|_{\ell^2(\Z^s,H)}^2+
\lim_{n\rightarrow \infty}
\left\|x_n-\sum_{\ell=1}^{m^\prime}T_{w^\ell_n}x^\ell\right\|_{\ell^2(\Z^s,H)}^2 \label{CC3}
\end{align}
for $m^\prime = 1, \ldots, m$,
\begin{equation}
\limsup_{n\rightarrow \infty}\left\|
x_n-\sum_{\ell=1}^mT_{w^\ell_n}x^\ell
\right\|_{\ell^\infty(\Z^s,H)}\leq \varepsilon, \label{CC4}
\end{equation}
and
\begin{equation}
\lim_{n\rightarrow \infty}\left\|x_n-T_{w^1_n}x^1
\right\|_{\ell^\infty(\Z^s,H)}=0 \label{CC5}
\end{equation}
if $m=1$.
Here the weak convergence is understood in $\ell^2(\Z^s,H)$ and \linebreak 
$T_w:\ell^2(\Z^s,H)\rightarrow \ell^2(\Z^s,H)$, $w\in \Z^s$, denotes the translation operator $T_w (x_j)=(x_{j-w})$.
\end{lemma}

\begin{proof}
Observe that $x_n \neq 0$ for each $n \in {\mathbb N}$ and
$$\lim_{n\rightarrow \infty}\|x_n\|_{\ell^2}\geq
\lim_{n\rightarrow \infty}\|x_n\|_{\ell^\infty}>0.$$
(In this proof we abbreviate
$\ell^2(\Z^s,H)$ and $\ell^\infty(\Z^s,H)$ to respectively $\ell^2$ and $\ell^\infty$
and extract subsequences where necessary for the validity of our arguments.)

Choose the sequence
$\{w_n^1\}\subset \Z^s$ such that  $\|x_{n,w_n^1}\|_H=\|x_n\|_{\ell^\infty}$
($\{x_n\} \in \ell^2(\Z^s,H)$ implies that $\|x_{n,j}\|_H \rightarrow 0$
as $|j| \rightarrow \infty)$.
Because $\{T_{-w^1_n}x_{n}\}$ is bounded there exists $x^1 \in \ell^2$ such that
$T_{-w^1_n}x_{n}\rightharpoonup x^1$, and the 
relative compactness of $S$ implies that $\lim_{n \rightarrow \infty}\|(T_{-w^1_n}x_{n})_j - x^1_j\|_H=0$
for each $j \in {\mathbb Z}^s$.
Since $\|(T_{-w^1_n}x_{n})_j\|_H \leq \|(T_{-w^1_n}x_{n})_0\|_H$ by construction it follows that
$\|x_j^1\|_H \leq \|x_0^1\|_H$ for each $j \in \Z^s$ and hence that $\|x_0^1\|_H = \|x^1\|_{\ell^\infty}$.
We conclude that $\|x^1\|_{\ell^\infty}=\lim_{n \rightarrow \infty} \|x_{n,w_n^1}\|_H=\lim_{n\rightarrow \infty}\|x_n\|_{\ell^\infty}>0$. 
Furthermore
\begin{align*}
\lim_{n\rightarrow \infty}\|x_n\|_{\ell^2}^2
& = \lim_{n \rightarrow \infty} \|T_{-w^1_n}x_n\|_{\ell^2}^2 \\
& =\lim_{n\rightarrow \infty}\|x^1+(T_{-w^1_n}x_n-x^1)\|_{\ell^2}^2 \\
&=\|x^1\|_{\ell^2}^2 + 2\underbrace{\lim_{n\rightarrow \infty} \langle x^1, T_{-w^1_n}x_{n}-x^1 \rangle_{\ell^2}^2}_{\displaystyle =0}
+\lim_{n\rightarrow \infty}\|T_{-w^1_n}x_{n}-x^1\|_{\ell^2}^2.
\end{align*}
If $\lim_{n\rightarrow \infty}\left\|x_n-T_{w^1_n}x^1\right\|_{\ell^\infty}=0$ we set $m=1$, concluding the proof. Otherwise
we apply the above argument to the sequence 
$\{x_n^{(2)}\}$, where $x_n^{(2)}=x_n-T_{w^1_n}x^1$ and proceed iteratively;
it remains to show that we can choose $m \geq 2$ such that \eqref{CC4} is satisfied.

Suppose that for each $m\geq 3$ there exist
vectors $x^1,\ldots,x^m\in \ell^2(\Z^s,H)$
and  sequences\linebreak
 $\{w^1_n\}$, \ldots, $\{w^m_n\} \subset \Z^s$ which satisfy \eqref{CC1}--\eqref{CC3} and
$$\lim_{n\rightarrow \infty}
\left\|x_n-\sum_{\ell=1}^{m^\prime}T_{w^\ell_n}x^\ell\right\|_{\ell^\infty}
>\varepsilon$$
for $m^\prime=2,\ldots,m$.
Choosing $m>1+\varepsilon^{-2}   \lim_{n\rightarrow \infty}\|x_n\|_{\ell^2}^2$, we obtain the contradiction
\[
\lim_{n\rightarrow \infty}\|x_n\|_{\ell^2}^2 \\
> \varepsilon^2+ \sum_{\ell=1}^m\|x^{\ell}\|_{\ell^2}^2
\geq  \varepsilon^2+ \sum_{\ell=3}^m\|x^{\ell}\|_{\ell^\infty}^2
> (m-1)\varepsilon^2>
\lim_{n\rightarrow \infty}\|x_n\|_{\ell^2}^2. \qedhere
\]
\end{proof}

\begin{lemma} \label{lem:cc2}
The sequences $\{w_n^1\}$, \ldots, $\{w_n^m\}$ satisfy
$$\lim_{n \rightarrow \infty} |w_n^{m^{\prime\prime}}-w_n^{m^\prime}| \rightarrow \infty, \qquad 1 \leq m^{\prime\prime} < m^\prime \leq m$$
so that in particular
$$
T_{-w^{m^\prime}_n}x_n\rightharpoonup x^{m^\prime}, \qquad m^\prime=1,\ldots,m.
$$
\end{lemma}

\begin{proof}
Suppose the result does not hold and select the smallest $m^\prime \in \{2,\ldots,m\}$
such that\linebreak
$|w^{m^{\prime\prime}}_n-w^{m^\prime}_n|\not \rightarrow \infty$ 
for some $m^{\prime\prime}\in\{1,\ldots,m^\prime-1\}$;
by a judicious choice of
subsequences we can arrange that $w^{m^\prime}_n-w^{m^{\prime\prime}}_n$
is equal to a constant $j \in \Z^s$.

On the one hand 
$\lim_{n\rightarrow \infty}|w^{\ell}_n-w^{m^{\prime\prime}}_n|=\infty$
for $\ell=1,\ldots, m^{\prime\prime}-1$, so that
\begin{equation}
T_{-w^{m^{\prime\prime}}_n}x_n\rightharpoonup x^{m^{\prime\prime}} \label{Trans cont}
\end{equation}
(see \eqref{CC1}), while on the other hand
$$T_{-w^{m^\prime}_n}\left(x_n-\sum_{\ell=1}^{m^\prime-1}T_{w^\ell_n}x^\ell\right)
\rightharpoonup x^{m^\prime},$$
so that
$$T_{-w^{m^{\prime\prime}}_n-j}x_n-\sum_{\ell=1}^{m^\prime-1}T_{-w^{m^{\prime\prime}}_n-j+w^\ell_n}x^\ell
\rightharpoonup x^{m^\prime}$$
and hence
$$T_{-w^{m^{\prime\prime}}_n}x_n-x^{m^{\prime\prime}}
\rightharpoonup T_jx^{m^\prime},$$
which contradicts \eqref{Trans cont} because $x^{m^\prime} \neq 0$.
\end{proof}

\noindent\\
{\bf Acknowledgements.} M. D. Groves would like to thank the Knut and Alice Wallenberg Foundation for funding a visiting professorship
at Lund University during which this paper was prepared.  E. Wahl\'{e}n was supported by the Swedish Research Council (grant no. 621-2012-3753).

\end{document}